\documentclass{article}
\usepackage{graphicx}
\usepackage[utf8]{inputenc}
\usepackage[T1]{fontenc}
\usepackage{amsfonts,amsmath,bbm,yfonts, etoolbox}
\usepackage{amssymb, enumerate, amscd,tikz-cd}
\numberwithin{equation}{section}
\usepackage{latexsym}
\usepackage{extarrows}
\usepackage{fullpage}
\usepackage{tikz}
\usetikzlibrary{hobby}
\usetikzlibrary{intersections}
\usepackage{caption}
\usetikzlibrary{matrix}
\usepackage{graphicx}
\usepackage{enumitem} 
\usetikzlibrary[patterns, decorations.pathreplacing]
\usetikzlibrary{shapes,trees, positioning} 
\tikzstyle{ann} = [fill=white,inner sep=6pt]
\usepackage{calligra}
\usepackage{calrsfs}
\usepackage[mathcal]{eucal}

\newcommand{\leftsquigarrow}{\mathrel{\rotatebox[origin=c]{-180}{$\rightsquigarrow$}}}
\newcommand{\longupdownarrow}{\mathrel{\rotatebox[origin=c]{90}{$\longleftrightarrow$}}}
\usepackage[a4paper,margin=3.1cm, headsep=24pt, headheight=2.6cm]{geometry}
\usepackage[colorlinks]{hyperref}
\usepackage[nameinlink]{cleveref}
\hypersetup{colorlinks,breaklinks,
	citecolor=[rgb]{0.1,0.7,0.5},
	linkcolor=[rgb]{0.8,0.2,0.1}}
\usepackage[all,cmtip]{xy}
\usepackage[sc, osf]{mathpazo}
\usepackage{graphicx}
\usepackage{float}
\usepackage{mathtools}
\DeclareMathOperator{\spec}{Spec}
\AtBeginDocument{\renewcommand{\hbar}{\hslash}}
\usepackage{ntheorem}
\makeatletter
\newtheoremstyle{mystyle}% 
{\item[\hskip \labelsep \theorem@headerfont ##1 \scshape{\bfseries{##2}}\theorem@separator]}
{\item[\hskip\labelsep \theorem@headerfont \scshape{\bfseries{##1}}  \scshape{\bfseries{##2}} \textbf{(##3)}\theorem@separator]}
\makeatother
\theoremstyle{mystyle}
\theoremheaderfont{\scshape\bfseries}
\theorembodyfont{\upshape}
\newtheorem{theorem}{Theorem}[section]
\newtheorem{mmtheorem}{Theorem}
 
\newtheorem{proposition}[theorem]{Proposition}
\newtheorem{question}{Question}[question]

\newtheorem{lemma}[theorem]{Lemma}
\newtheorem{remark}[theorem]{Remark}
\newtheorem{example}[theorem]{Example}
\newtheorem{cor}[theorem]{Corollary}
\newtheorem{defn}[theorem]{Definition}

\newtheorem{convention}[theorem]{Convention}
\newtheorem{notation}[theorem]{Notations}
\newtheorem{claim}[theorem]{Claim}
\newenvironment{proof}
{\textsc{Proof:} }
\newcommand\Simp{\mathit{Simp}}
\newcommand\Sin{\mathit{Simp_n}}
\newcommand\Frob{\mathrm{Frob}}
\newcommand\Sym{\mathit{Sym}}
\newcommand\Br{\mathrm{Branch}}
\newcommand\Trace{\mathrm{Trace}}
\newcommand\Aut{\mathrm{Aut}}

\newcommand\M{\mathrm{M_n}}

\newcommand{\etale}{étale \,}

\newcommand\longmapsfrom{\mathrel{\reflectbox{\ensuremath{\longmapsto}}}}

\def\P{\mathcal{P}}
\def\A{\mathbb{A}}
\def\E{\mathcal{E}}
\def\Q{\mathbb{Q}}
\def\et{\acute{e}t}
\def\length{\mathit{length}}

\def\F{\mathbb{F}}

\def\X{\mathcal{X}}
\def\I{\mathcal{I}}
\def\S{\mathcal{S}}
\def\U{\mathcal{U}}
\def\Z{\mathcal{Z}}
\def\K{\mathit{K}}

\date{}
\usepackage{relsize}
\newcount \questionno
\def\question{\medbreak
	\global \advance \questionno 1
	\textbf{Question \the\questionno}.\enspace \ignorespaces}
\newcommand{\remove}[1]

\title{Cohomology of the space of polynomial maps on $\A^1$ with prescribed ramification}
\author{Oishee Banerjee}

\DeclareUnicodeCharacter{200E}{}
\begin{document}
						
			\maketitle
		\begin{abstract}
		In this paper we study the moduli spaces $\Simp^m_n$ of degree $n+1$ morphisms $ \mathbb{A}^1_{\K} \to \mathbb{A}^1_{\K}$  with "ramification length $<m$" over an algebraically closed field $\K$. For each $m$, the moduli space $\Simp^m_n$ is a Zariski open subset of the space of degree $n+1$ polynomials over $\K$ up to $Aut (\A^1_{\K})$. It is, in a way, orthogonal to the many papers about polynomials with prescribed zeroes- here we are prescribing, instead, the ramification data. Exploiting the topological properties of the poset that encodes the ramification behaviour, we use a sheaf-theoretic argument to compute $H^*(\Simp^m_n(\mathbb{C});\mathbb{Q})$ as well as the \etale cohomology $H^*_{\et}({\Simp^m_n}_{/\K};\mathbb{Q}_{\ell})$ for $char\K=0$ or $char\K> n+1$. As a by-product we obtain that $H^*(\Simp^m_n(\mathbb{C}); \mathbb{Q})$ is independent of $n$, thus implying rational cohomological stability. When $char \K>0$ our methods compute $H^*_{\et}(\Simp^m_n;\mathbb{Q}_{\ell})$ provided $char\K>n+1$ and show that the \etale cohomology groups in positive characteristics do not stabilize.
	\end{abstract}

\section{Introduction}
We work throughout over an algebraically closed field $\K$. Let $$\M: = \{f: \A^1 \to \A^1: f \text{  is a morphism of degree } n+1\}/ \Aut(\A^1).$$ 

\noindent We identify $\M$ with the space of all degree $n+1$ monic polynomials over $\K$ that vanish at $0$.
There exists vast literature studying subvarieties of $\M$, e.g.  the space of square-free polynomials (i.e configuration spaces of distinct points), or the space of finite morphisms to $\A^1$ with a fixed a Galois group $G$ etc. In this paper we consider a natural but quite different problem by considering the subvarieties $\Simp^m_n \subset \M$ of morphisms with "total ramification $<m$".

\bigskip \noindent To be precise, let $$\mathcal{N} := \{\text{finite subsets of (not necessarily distinct) integers } \geq 2\}.$$
For $f \in \M$, let $v_f(a)$ denote the \textit{valuation of $f$ at $a$} (for a definition see Section 2 below). Let $$Ram(f) :=\{a\in\A^1: v_f(a)\geq 2\}$$ be the set of ramification points of $f$. If $a$ is a ramification point of $f$, define the \emph{ramification index of $f$ at $a$} to be the positive integer $v_f(a)$. Let $\Br (f) := f(Ram(f)) \subset \A^1$ be the set of \emph{branch points.} 
A branch point $b\in \A^1$ determines $B_b(f)\in \mathcal{N}$ via $$B_b(f):= \{\text{ramification indices of elements of }f^{-1}(b)\}.$$ Let $$l(B_b(f)) := \sum_{e\in B_b(f)}(e-1)$$ be the \emph{ramification length over $b$}". The \emph{total ramification length of $f$} is \begin{equation}\label{eq1.1}
\length(f):=\sum\limits_{b\in \Br (f)}\big( l(B_b(f))-1\big).
\end{equation}Let $$\Simp^m_n : =\{f\in\M: \length(f)<m\}.$$ This is a Zariski open dense subset of $\M$, and hence a smooth variety over $K$. In fact, as we shall soon see, $\Simp^m_n$ is the complement of a locus defined by polynomials with coefficients in $\mathbb{Z}$, and hence is a reduced separated scheme of finite type over $\mathbb{Z}$.
When $m=1$, we get the locus of \emph{simply-branched polynomials}, which we denote by $\Sin.$ These are the degree $n+1$ morphisms $f:\A^1\to \A^1$ with \emph{simple branch points.}
\begin{remark}
	Note that  $\{f\in\M: \length(f)=m\}$, the locally closed stratum of  polynomials with total ramification length $m$,  has codimension $m$ in $\M$ by the Riemann-Hurwitz formula. In other words, \eqref{eq1.1} is the bridge that relates the codimension of a stratum with the total ramification length of polynomials in that stratum, via the Riemann-Hurwitz formula. 
	
\hfill$\square$\end{remark}
		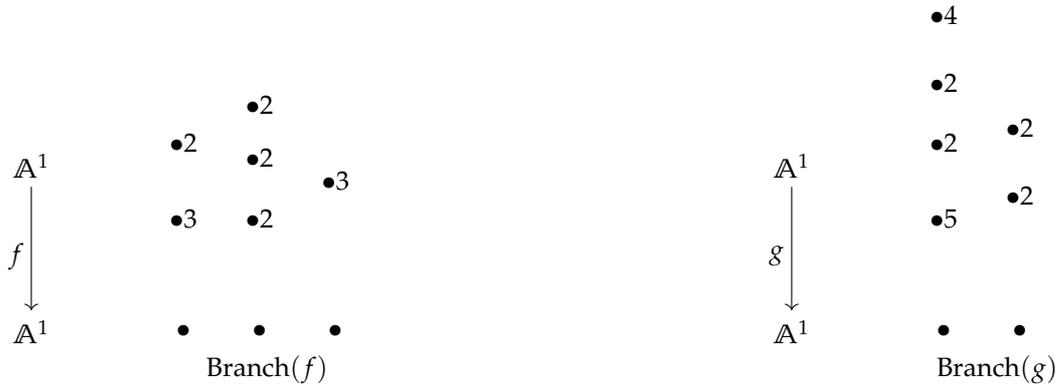
\begin{figure}[H]
	\begin{tikzpicture}
	\node at (0,0) {$\bullet 2$};               	\node at (10,0) {$\bullet 2$};
	\node (A1) at (-2,-0.3) {$\A^1$};            	\node (A3) at (8,-0.3) {$\A^1$};
	\node (A2) at (-2,-2.5) {$\A^1$};              \node (A4) at (8,-2.5) {$\A^1$};
                                                                   	\node at (7.8,-1.5) {$g$};
	\draw [->] (A1)-- (A2);                              \draw [->] (A3)-- (A4);
	\node at (0,-1) {$\bullet 3$};                    	\node at (10,-1) {$\bullet 5$};
	\node at (-2.2,-1.5) {$f$};                        \node at (10,-2.5) {$\bullet$};
                                                              	\node at (11,-2.5) {$\bullet$};
	\node at (1,0.5) {$\bullet 2$};                 	\node at (10,0.8) {$\bullet 2$};                         
	\node at (1,-0.2) {$\bullet 2$};                	\node at (10,1.7) {$\bullet 4$};    
	\node at (1,-1) {$\bullet 2$};
	\node at (0,-2.5) {$\bullet$};                       	\node at (11,-0.7) {$\bullet 2$}; 
	\node at (1,-2.5) {$\bullet$};                      	\node at (11,0.2) {$\bullet 2$}; 
	\node at (2,-2.5) {$\bullet$};
	\node at (2,-0.5) {$\bullet 3$};
		\node at (1.1,-3) {$\Br(f)$};
			\node at (10.7,-3) {$\Br(g)$};
	\end{tikzpicture}	\caption{A schematic of the ramification points (with indices specified) over the branch points of two morphisms $f, g \in \mathit{Simp^7_n}$ for a fixed $n\geq 13.$}
	\end{figure}
\noindent Let $\mathbf{p}(N)$ denotes the number of partitions of a positive integer $N$. Let $\mathbf{c}:\mathbb{Z}^{+}\to \mathbb{Z}^{+}$ be defined via \begin{align}
\mathbf{c}(m) =\mathlarger{\mathlarger{‎‎\sum}}\limits_{k\geq 1}\Bigg(\mathlarger{\mathlarger{‎‎\sum}}\limits_{\substack{n_1+\ldots +n_k =m\\ n_1\leq \ldots \leq n_k}} \mathbf{p}(n_1+1)\ldots \mathbf{p}(n_k+1)\Bigg).\label{c(m)}
	\end{align}By $H^i$ (respectively $H^i_{\et}$) we will mean singular (respectively \etale) cohomology. If $V$ is a $\mathbb{Q}_{\ell}$ vector space and if $m\in \mathbb{Z}$ then we let $V(m)$ denote the \emph{$m^{th}$ Tate twist of $V$}.
 Our main theorem computes the cohomology of $\Simp^m_n.$ \begin{mmtheorem}{\label{thmB}}
	Let $m,n \geq 1$. Then the following hold. \begin{enumerate}
		\item For all $n\geq 3m$: \[
		H^i(\Simp^m_n(\mathbb{C});\Q) = \Bigg\{\begin{array}{lr}
		\Q & \text{for } i=0, \\
		\Q^{\oplus \mathbf{c}(m)} & \text{for } i=m,\\
		0 & \text{otherwise.} 
		\end{array}	\]		
		\item  Let $\kappa$ be a field satisfying $char\, \kappa >n+1$ or $char\, \kappa =0$. Then for all $n\geq 3m$, we have the following isomorphism of $Gal({\overline{\kappa}/ {\kappa}})$-representations: \[
		H^i_{\et}({\Simp^m_n}_{/\overline{\kappa}};\Q_{\ell}) = \Bigg\{\begin{array}{lr}
		\Q_{\ell}(0) & \text{for } i=0, \\
		\Q_{\ell}(-m)^{\oplus \mathbf{c}(m)} & \text{for } i=m,\\
		0 & \text{otherwise,} 
		\end{array}	\]
		whenever $\ell$ is prime to $char \,\kappa$. 
	\end{enumerate} \hfill$\square$
\end{mmtheorem}

\bigskip	\noindent \textbf{An arithmetic application.} Theorem \ref{thmB} paired with the Grothendieck-Lefschetz fixed point theorem gives us the following:
\begin{cor}{\label{corB}}
Let $m,n\geq 1$ and let $q =p^d$, where $p$ is a prime and $d\geq 1$. Then $$\# \Simp^m_n(\F_q) = q^n -\mathbf{c}(m)q^{n-m}$$ for all $n<p-1$ and $m\leq \frac{n}{3}$.

 \hfill$\square$
\end{cor}
\begin{remark}\label{rem1.3}
The case $m=1$ itself is of special interest- it answers questions about the topology of the moduli space of simply-branched morphisms. If $m=1$ then $\mathbf{c}(m) =2$. So when $n\geq 3$, Theorem \ref{thmB} provides answers for $H^*(\Simp_n(\mathbb{C});\mathbb{Q})$ and $H^*({\Simp_n}_{/\K};\mathbb{Q}_{\ell})$ which were not previously known. In particular, for all $n\geq 3$ Corollary \ref{corB} immediately implies the following:$$\#\Simp_n(\mathbb{F}_q)  = q^n-2q^{n-1}$$ where $q=p^d$, provided $n+1<p$. When $n= 1$, Them \ref{thmB} is trivial because all degree $2$ morphisms are simply-branched i.e. $\Simp_1\cong \A^1$. If $n=2$, then $\Simp_2$ is isomorphic to the space of square-free quadratic polynomials by the map defined \eqref{Disom}. Results for the latter space are well-known thanks to Arnol'd's work (see e.g. \cite{Arnol'd1969}). 
\end{remark}
\paragraph{Further remarks.}
\begin{enumerate} 
	\item In  characteristic $p>0$, we could have also considered the moduli space of polynomials of degree $n+1$ which are unramified as self-maps of the affine line. However, we have also seen that these spaces are nonempty if and only if $n+1 =p^k$ for some $k$. So, our assumption of $n+1<p$ rules out the unramified case.
	\item Note that Item (2) in Theorem \ref{thmB} does not imply \etale cohomological stability when $char \K >0$.  When $n$ is large, morphisms with wild ramification will inevitably come into the picture. Via Artin-Schreier theory one can construct infinite families of degree $n$ morphisms $f:\A^1_{\overline{\F}_p} \to\A^1_{\overline{\F}_p}$ with a fixed ramification type.  Furthermore, note that $\Simp^m_n$ is not a proper scheme over $\mathbb{Z}$. So even the customary base change and Grothendieck-Lefschetz theorems would not help with finding a formula for $\# \Sin(\F_q)$ for large values of $n$.
\end{enumerate}
\paragraph{Some context.}
\begin{enumerate}
	\item Theorem \ref{thmB} is orthogonal to the plethora of results concerning the (co)homology of the moduli space of polynomials with a prescribed order of zeroes (also known as configuration spaces on $\mathbb{C}$) due to Arnol'd (see e.g ~\cite{Arnol'd1969,V.I.Arnold1970}), Napolitano (\cite{Napolitano1998}) etc. While most results concentrate on spaces recording the zeroes of polynomials, $\Simp^m_{n}$ records the ramification. This in turn prevents us quoting the Leray Serre spectral sequence for inclusion, unlike the papers on configuration spaces.
	In fact, our results should be viewed in the spirit of the long standing open problem of understanding the topology of the  \emph{Hurwitz space}. The irreducibility of the Hurwitz space is a classical result proved in ~\cite{Clebsch1872}, with a more modern account in ~\cite{Fulton1969}, but the topology of its subvarieties corresponding specific ramification loci is almost completely unknown. Our result is that of stability of the cohomology of these Hurwitz spaces satisfying certain conditions.
	
	\item  A well-known method of looking at this the Hurwitz spaces, at least when $\K = \mathbb{C}$, is by considering topological finite covers of punctured smooth projective curves (see, e.g. \cite{Romagny2006}, and the references therein).  As an example, note that each element of $\M$ corresponds to an $(n+1)$-sheeted cover of $\A_{\mathbb{C}}^1 - \{p_1,\ldots, p_n\})$. One looks at finite quotients of the topological fundamental group  $\pi_1^{top}(\A^1-\{p_1,\ldots,p_n\})$ (which is finitely generated), or in turn, subgroups of $\pi_1^{top}(\A^1-\{p_1,\ldots,p_n\})$ of a fixed (finite) index, of which there are only finitely many. 
	
	In a beautiful paper on Cohen-Lenstra statistics, Ellenberg-Venkatesh-Westerland study Hurwitz schemes with fixed Galois group (see \cite{Ellenberg2015}) and prove a homological stability result.  The resemblance of  $\Simp^m_n$ with the Hurwitz schemes in \textit{loc. cit}. is close enough to warrant digging a little deeper to see why the techniques in ~\cite{Ellenberg2015} seem unlikely to imply Theorem \ref{thmB}. The key difference between this paper and theirs lies in the Galois groups of the finite covers of $\A^1-\{p_1,\ldots,p_n\}$. In ~\cite{Ellenberg2015}, they consider $G$-covers of $\A_{\mathbb{C}}^1 - \{p_1,\ldots, p_n\}$, where $G$ is a fixed group (satisfying certain conditions), and the number of branch points grow, thereby increasing the genus of the projective completion of the cover but keeping the degree of the cover unchanged. In our case, the genus of the cover is always $0$, whereas the monodromy group, which in some cases would turn out to be $\mathfrak{S}_{n+1}$ (e.g when all branch points are simple) grows with the degree of the cover. 
	
	\item In contrast to Theorem \ref{thmB}, the \etale cohomology groups $H^i_{\et}(\Simp^m_n;\mathbb{Q}_{\ell})$ do not stabilize when $char \K>0$- a divergence from other comparable stability results (see e.g \cite{Ellenberg2015}, and Farb-Wolfson's work on configuration spaces see \cite{Farb2015}). Indeed, the moduli space of polynomials $f\in \overline{\F}_p[x]$ of degree $n$ that are unramified as self-maps of $\A^1_{\overline{\F}_p}$ is nonempty if and only if $n$ is a prime power. To see this, note that when $n$ is a prime power, there are the Artin-Scherier examples like $x^n-x$, which is unramified since $\frac{d}{dx} (x^n-x) =-1 \neq 0$; the other direction follows from the work of Grothendieck (see \cite{Grothendieck1971}) which goes roughly in the following way. Let $\phi$ be a polynomial of degree $n$ where $n=p^km$ for some $m$ and $p \not\vert \, m$, and suppose the finite morphism $\phi: \mathbb{P}^1_{\overline{\F}_p}\to \mathbb{P}^1_{\overline{\F}_p}$ is ramified only at $\infty$. Its tame pullback kills the prime-to-$p$ part of the inertia at $\infty$ and gives us an unramified morphism $\widehat{\phi}: \A^1_{\overline{\F}_p}\to \A^1_{\overline{\F}_p}$ with tame ramification at $\infty$ and the inertia group cyclic of order $m$. Since $\text{gcd}(m,p)=1$, the map $\widehat{\phi}$ can be lifted to characteristic $0$, which then forces $m= 1$.  
	
	The fact that $H^i_{\et}({\Simp^m_{n}}_{/ \overline{\F}_p};\Q_{\ell})$ does not stabilize is a manifestation of Abhyankar's philosophy: that prime-to-$p$ situation mimics the characteristic $0$ picture, else, every type of cover that can possibly occur, indeed occurs (see \cite[Section 3]{Stevenson2017}).
	
\end{enumerate}

\paragraph{Outline of proof of Theorem \ref{thmB}.}
Fixing $n\geq 3$, our approach to computing $H^i(\Simp^m_{n})$ for each $m \geq 0$ can be summarized as follows.
	\begin{enumerate}
		\item We first relate the ramification of a polynomial with its derivative. We relocate the whole problem to $M'_n$, the space of derivatives of all elements in $\M$. Noting that $M'_n \cong \M$, we reduce the problem to computing the cohomology of the image of $\Simp^m_n$ in $M'_n$ by studying the  ordered zeroes of elements in $M'_n$ i.e. "the root cover of $M'_n$".
	\item  We construct  posets that encode the ramification behaviour of elements of $\M$. More precisely, fixing the ramification data stratifies $\M$, and in turn $M'_n$, into a disjoint union of locally closed subsets whose closures give us a covering of the root cover of $M'_n$ by closed sets.  Their pre-image in the root cover is combinatorially described by the posets $\P^m_n$. Our first step is to prove that $\P^m_n$ is \emph{shellable} (see \Cref{Section3.1}). The key implication of being shellable, for us, is that the only nonzero reduced cohomology of an "open interval" in $\P^m_n$ resides in its top dimension.
	\item We study the geometric properties of the strata in the above-mentioned stratification in \Cref{Section4}, in particular Proposition \ref{VBness}.
	\item We use shellability of $\P^m_n$ to construct the resolution \eqref{resol} of $j_{!}\Q_{\U^m_n}$ (see Lemma \ref{lem5.1}) where $\U^m_n$ denotes the space of "ordered ramification points", defined in \eqref{Un}. 
	\item Finally, we compute  $H^*(\U^m_n,\Q)$ by incorporating  the geometric  properties of the stratification from Item 2, and shellability of $\P^m_n$ from Item 1 in the  resolution of $j_{!}\Q_{\U^m_n}$, as mentioned in Item 3. Taking $\mathfrak{S}_n$ of the resulting spectral sequence now finishes the proof of Theorem \ref{thmB} since most terms on the $E^1$ page turn out to be $0$ thanks to Propositions \ref{VBness} and \ref{semimodular} (see \Cref{Section5}).
\end{enumerate}

	\subsection*{Acknowledgements}
I am very grateful to my advisor, Benson Farb, for his patient guidance and unconditional support.  His many invaluable comments on earlier versions of this paper have been instrumental in its improvement. I thank Patricia Hersh for pointing out some existing results about posets, Alexander Beilinson, a brief discussion with whom turned out to be crucial to my understanding of the final section of this paper, and Akhil Mathew for indicating the right source to support a proof in Section 5. I heartily thank Lei Chen and Ronno Das for being my sounding board during the final days of this project. I also thank Madhav Nori for posing a question during a discussion which eventually led to this problem.

\section{Ramification, derivatives and the "ramification  cover"}
In this section we elaborate on our first step discussed in the proof outline above. We assume that $char\K=0$ or $char\K >n+1$. We show that the ramification behaviour of a polynomial is reflected, to a large extent, by its derivative.  Switching to the "space of derivatives" is our first step to prove Theorem \ref{thmB}.

For convenience, let us briefly recall (and expand on) the definitions from page 1 of the introduction. Let $f:X\to Y$ be a finite morphism of smooth curves defined over $\K$. Let $f^{\#}: \mathcal{O}_{Y,b}\to \mathcal{O}_{X,a}$ be the homomorphism induced by $f$ on the stalks of the structure sheaves $\mathcal{O}_Y$ and $\mathcal{O}_X$ at the closed points $b=f(a) \in Y$ and $a\in X$ respectively. Let $y$ be a generator for the maximal ideal in $\mathcal{O}_{Y,b}$. The valuation of $f$ at $a$,  which we denote by $v_f(a)$, is defined as $v_a(f^{\#}(y))$ where $v_a$ is the valuation associated to the discrete valuation ring $\mathcal{O}_{X,a}$. In this paper we assume $X = Y = \A^1$. \begin{defn}[{\bf{Ramification data}}]\label{ramdata}
	Let $n$ be a positive integer. For an element $\phi \in \M$ we   define the  \emph{ramification data of $\phi$}  as three sets of data:
\begin{enumerate}
\item the ramification points of $\phi$, given by $Ram(\phi) = \{a\in \A^1: v_{\phi}(a)\geq 2\}$,
\item the branch points of $\phi$, given by $\phi(Ram(\phi))$,
\item associated to each point $b\in \Br(\phi)$ we define the \emph{ramification of $\phi$ over $b$} as an unordered $l(B_b(\phi))$-tuple $Ram_b(\phi) \in \Sym^{l(B_b(\phi))}\A^1$ via:
 \begin{equation}\label{eq2.1}
Ram_b(\phi) := \{a\in \A^1:a\in \phi^{-1}(b) \cap Ram(\phi), \text{ counted } (v_\phi(a)-1) \text{ times} \}
 \end{equation}

\end{enumerate} \hfill $\square$
\begin{defn}\label{simplybranchedRam}
For $a\in Ram(\phi)$ we say that $\phi$ is \emph{simply-branched} at $a$ or $a$ is a  \emph{simply-branched ramification point} of $\phi$ if $B_{\phi(a)}(\phi) = \{2\}$.  
We that $b\in \Br(\phi)$ is a \textit{simple branch point} of $\phi$ or $\phi$ is \emph{simply-branched} at $b$ if $b$ is the image of a simply-branched ramification point of $\phi$.

\noindent For $a_1, a_2 \in Ram(\phi)$ we say that \emph{$a_1$ and $a_2$ are sibling ramification points} if $\phi(a_1) = \phi(a_2)$.

\hfill $\square$ \end{defn} 

\begin{remark}
Note that if $\phi$ is simply-branched at $a$, it is clearly simply ramified at $a$, but the converse is not true.	The above definition also implies that $\phi$ is\emph{ non-simply-branched at $a\in Ram(\phi)$} if and only if $l(B_{\phi(a)}(\phi)) \geq 2$.
\end{remark}
\end{defn}
Let the ramification data of $\phi \in \M$ be given by:
\begin{equation}\label{FR}\begin{split}
	& \Br(\phi) = \{b_1,\ldots, b_p\}\\
	& \text{for each }  i, \text{ let }B_{b_i}(\phi) = \{e^{1}_i, \ldots, e^{k_i}_i \},\\ &\text{and let } Ram_{b_i}(\phi) = \Big(\underbrace{a^{1}_i,\ldots, a^{1}_i}_{e^i_1-1}, \ldots,\underbrace{a^{k_i}_i,\ldots, a^{k_i}_i}_{e^i_{k_i}-1}\Big).\end{split}
\end{equation} 
%\begin{equation}
%\begin{aligned}
%&  \text{ 1. Let }\Br(\phi) = \{b_1,\ldots, b_p\},\\ &  \text{ 2. for each } i, \text{ let }B_{b_i} = \{e^i_{1}, \ldots, e^i_{k_i} \}\subset \mathcal{N}\\ &\text{ 3. and let } Ram(\phi) \cap \phi^{-1}(b_i) = \{a^i_{1}, \ldots, a^i_{k_i}\}
%\end{aligned}\label{ramdata}
%\end{equation}
	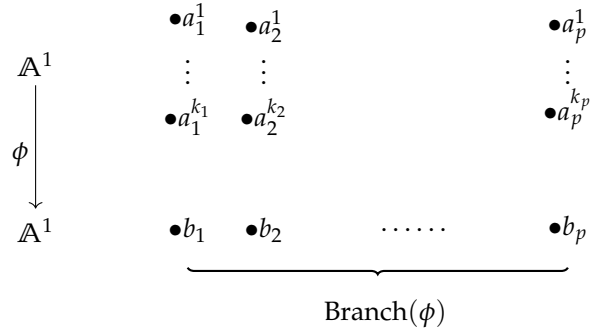
\begin{figure}[H]
		\centering
	\begin{tikzpicture}
	\node at (0,0.3) {$\bullet a^1_{1}$}; 
		\node at (1,0.2) {$\bullet a^1_{2}$}; 
		\node at (0,-0.3) {$\vdots$};               
		\node at (1,-0.3) {$\vdots$};                
	\node (A1) at (-2,-0.3) {$\A^1$};            
	\node (A2) at (-2,-2.5) {$\A^1$};            
	\draw [->] (A1)-- (A2);                         
	\node at (0,-1) {$\bullet a^{k_1}_1$};                    
	\node at (-2.2,-1.5) {$\phi$};                        \node at (5,-2.5) {$\bullet b_p$};
	\node at (1,-1) {$\bullet a^{k_2}_2$};                      \node at (5,-0.3) {$\vdots$};  
	\node at (0,-2.5) {$\bullet b_1$};      
		\node at (3,-2.5) {$\dots\dots$};                 	\node at (5,-0.9) {$\bullet a^{k_p}_p$}; 
	\node at (1,-2.5) {$\bullet b_2$};                      	\node at (5,0.2) {$\bullet a^1_{p}$}; 
	\node at (2.6,-3.6) {$\Br(\phi)$};
	\draw [
	thick,
	decoration={
		brace,
		mirror,
		raise=0.5cm
	},
	decorate
	] (0,-2.5) -- (5,-2.5) ;
	\end{tikzpicture}
	\caption{The diagram above is a schematic of the morphism $\phi\in \M$ with ramification data given by \eqref{FR}.}
\end{figure}
   Therefore, we have $$\phi'(x)  = \prod\limits_{1\leq i\leq k_1}(x-a^i_{1})^{e^i_{1} - 1} \ldots \prod\limits_{1\leq i\leq k_p}(x-a^i_{p})^{e^i_{p} - 1} $$ i.e. the derivatives of the morphisms in $\M$ completely determine, and are determined by the ramification indices. For $a \in Ram(\phi)$,  we define the \emph{differential length of $\phi$ at $a$} to be the order of vanishing of $\phi'$ at $a$ i.e if $v_{\phi}(a) =e$, then the differential length of $\phi$ at $a$ is $e-1$. \footnote{The differential length goes by other similar names, like, for example, \emph{length, different} etc. Our definition holds only for tamely ramified morphisms. For a general definition, see, e.g \cite{Hartshorne1977}.}
\noindent This leads us to introduce a new space  defined by $$M'_n := \Big\{\frac{1}{n+1}\phi': \phi\in \M\Big\}.$$ Note that $M'_n$ is the space of all monic degree $n$ polynomials over $\K$, and so $M'_n \cong \A^n$. Define the function \begin{gather*}
\mathcal{I}: M'_n \to \M \\ f \mapsto (n \times \text{the antiderivative of f that vanishes at } 0)
\end{gather*}
The Riemann-Hurwitz formula guarantees that the sum total of the differential lengths for any morphism $\phi\in \M$ is $n$. This gives us the following isomorphism:

\begin{align}\label{Disom}
\begin{split}
\mathcal{D}: \M \xlongrightarrow{\cong} M'_n \\
\phi \longmapsto \frac{\phi'}{n+1} \\
\I (f) \longmapsfrom f\\
\end{split}
\end{align}
Furthermore, let $\ell$ denote a positive integer, $\phi\in\M$ and $f\in M'_n$. Then it follows from \eqref{Disom}:  \begin{align}
& a\in Ram(\phi), \text{ with differential length } \ell \rightsquigarrow a\in \Big\{\text{Zeroes of }{\frac{\phi'}{n+1}} \Big\}\text{, with multiplicity }\ell\nonumber \\
& a\in Ram(\I(f)) , \text{ with differential length } \ell \leftsquigarrow a\in \big\{\text{Zeroes of }f \big\}\text{, with multiplicity }\ell
\end{align}

\noindent To study the topology of $\Simp^m_n$, we appeal to the isomorphism in \eqref{Disom} and define $$\S_n^m := \mathcal{D}\Simp_{n}^m.$$
So, $\S_n^m$ is Zariski open dense of $M'_n$ for all $m$. As with $\Simp_n$, we omit $m=1$ and write $\S_n$ instead of $\S^1_n$. We thus  have the following commutative diagram:

\[\begin{tikzcd}
\Simp_n^m \arrow{r}{\cong} \arrow[hookrightarrow]{d} & \S_n^m \arrow[hookrightarrow]{d} \\
\M \arrow{r}{\cong} & M'_n
\end{tikzcd}
\]

\paragraph{The ramification cover of $\M$.}

The \emph{ramification cover of $\M$} is the space of ordered ramification points of elements in $\M$, with multiplicities equal to  the differential lengths. In other words, the ramification cover of $\M$ is merely the root cover of $M'_n$. 	Let this be denoted by  $\X_n$. Therefore,  $\X_n = \mathbb{A}^n$ since it is the space of ordered roots of degree $n$ monic polynomials over $\K$. There is an obvious action of  $\mathfrak{S}_{n}$, the symmetric group on $n$ letters,  on $\X_n$ given by permuting the coordinates. This action is fixed-point free off the diagonals, resulting in a finite surjective morphism

\begin{gather}
\pi: \X_n \longrightarrow \X_n/\mathfrak{S}_n = M'_n \nonumber\\
(a_1,\ldots,a_n) \mapsto (x-a_1)\ldots (x-a_n)
\end{gather}\label{Xn}

The branch locus of $\pi$ is precisely the complement of the space of monic square-free degree $n$ polynomials in $M'_n$.  In other words, $\X_n$ is what one calls the "root-cover" of $M'_n$. Let $$U^m_n: = \pi^{-1} \S^m_n.$$ Thus, for example, the pre-image of $\S_n \subset M'_n$ in  $\X_n$ is given by \begin{gather}
\U_n := \pi^{-1}(\S_n) =\nonumber\\ \Bigg\{(a_1, \ldots, a_n): a_i \neq a_j \hspace{2mm} \text{and} \hspace{2mm} \Big(\I( \pi(a_1,\ldots, a_n))\Big)(a_i) \neq \Big(\I(\pi(a_1,\ldots, a_n)\Big)(a_j) \hspace{2mm} \forall i<j\Bigg\},
\end{gather}\label{Un}
where we write $\U_n$ instead of $\U_n^1$. 
Thus we have: 

\[\xymatrixrowsep{0.5pc}
\xymatrixcolsep{1pc}
\xymatrix{
	\U^m_n \ar@{->>}[dr]^{\pi\vert_{\U^m_n}} \ar@{^{(}->}[dd]\\
	& \S^m_n \ar@{^{(}->}[dd]   && \Simp^m_n \ar@{^{(}->}[dd] \ar[ll]^{\cong}\\
	\X_n \ar[dr]^{\pi}  \\
	& M'_n && \M \ar[ll]^{\cong}
}
\]
	\section{Stratification of $\X_n$: the combinatorics}
We should, for clarity, recall the convention fixed at the beginning of the introduction: all varieties are defined over an algebraically closed $\K$ and  
$n$ always denotes a positive integer that satisfies $n+1< char \K$ whenever $char \K >0$. 
We have seen briefly seen how ramification data on polynomials give us a stratification of the moduli space $\M$ and in turn, a stratification of $M'_n$.  The goal of this section is to encode this stratification into concrete combinatorial terms to give a stratification of the ramification cover $\X_n$. 

Now, $$\X_n-\U_n = \bigcup\limits_{1\leq i<j\leq n} T_{ij} \hspace{2mm} \bigcup \bigcup\limits_{1\leq i<j\leq n}D_{ij}$$ where
\begin{equation}
\begin{aligned}
&T_{ij}= \Bigg\{(a_1, \ldots, a_n): a_i = a_j\Bigg\}, \text{ and }\\ & 
D_{ij}=  \Bigg\{ (a_1, \ldots, a_n): \displaystyle \frac{\Big(\I (\pi(a_1,\ldots , a_n))\Big)(a_i)- \Big(\I( \pi (a_1,\ldots, a_n))\Big)(a_j)}{{ (a_i-a_j)^3}}=0 \Bigg\}
\end{aligned}
\end{equation}
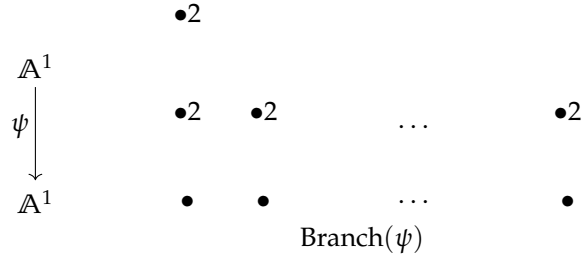
\begin{figure}[H]\centering\begin{tikzpicture}
	\node at (0,-1.3) {$\bullet$$3$};     
\node (A1) at (-2,-1.3) {$\A^1$};            
\node (A2) at (-2,-2.5) {$\A^1$};            
\draw [->] (A1)-- (A2);                         

\node at (-2.2,-1.8) {$\phi$};                        \node at (5,-2.5) {$\bullet$};    \node at (1,-2.5) {$\bullet$};   	\node at (3,-1.5) {$\dots$};   
\node at (1,-1.3) {$\bullet 2$}; 
\node at (5,-1.3) {$\bullet 2$};
\node at (0,-2.5) {$\bullet$};      
\node at (3,-2.5) {$\dots$};                   
                  \node at (2.3,-3) {$\Br(\phi)$};
\end{tikzpicture}\caption{The above diagram is a schematic of a generic point $\phi \in \I_{\circ}\pi(D_{ij}) \subset \M$. The one below is that of a generic point $\psi \in \I_{\circ}\pi(T_{ij})$.}
\vspace{3mm}
\centering	\begin{tikzpicture}
	\node at (0,-1.3) {$\bullet 2$};      	\node at (0,0) {$\bullet 2$}; \node at (3,-1.5) {$\dots$};  
	\node (A1) at (-2,-0.7) {$\A^1$};            
	\node (A2) at (-2,-2.5) {$\A^1$};            
	\draw [->] (A1)-- (A2);                         
	
	\node at (-2.2,-1.5) {$\psi$};                        \node at (5,-2.5) {$\bullet$};    \node at (1,-2.5) {$\bullet$};     
	\node at (1,-1.3) {$\bullet 2$}; 
	\node at (5,-1.3) {$\bullet 2$};
	\node at (0,-2.5) {$\bullet$};      
	\node at (3,-2.5) {$\dots$};                   
	\node at (2.3,-3) {$\Br(\psi)$};
	\end{tikzpicture}
\end{figure}
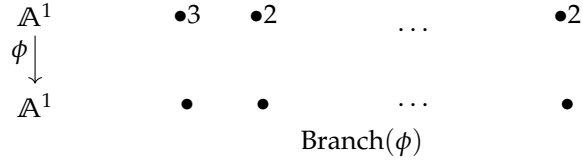
Note that $D_{ij}$, $T_{ij}$ as well as $\pi(D_{ij})$ and $\pi(T_{ij})$ are all $\mathbb{Z}$-schemes; they are defined in $\X_n$ and $M'_n$ respectively, each isomorphic to $\A^n$, by equations with coefficients in $\mathbb{Z}$. So $\U^m_n$, its $\mathfrak{S}_n$-quotient $\S^m_n$ and in turn $\Simp^m_n$ are defined over $\mathbb{Z}$ for all $m$ and $n$. Note that \begin{gather*}
\pi(D_{ij}) =\Big\{f \in M'_n: \I(f) \text{ has exactly one branch point that satisfies } B_b(\I(f)) =\{2, 2\}  \Big\}\\ \text{and, } 
\pi(T_{ij}) =\Big\{f \in M'_n: \I(f) \text{ has exactly one branch point that satisfies } B_b(\I(f)) =\{3\}  \Big\},
\end{gather*} as show in the diagram above (also, see Definition \ref{ramdata} and \eqref{Disom}.)
The closed subvarieties formed by the intersection of various combinations of the divisors $D_{ij}$ and $T_{ij}$ give us a stratification of $\X^n$,  the combinatorics of which we describe now. 

\subsection{A combinatorial description of  stratification by locally closed subsets}\label{Section2.1}
In this section, we describe $\X_n$ as a disjoint union of locally closed subsets, are indexed by a the elements of a certain poset. We first fix a convention: if $\rho$ is a partition of the set $\{1,2, \ldots, n\}$, we denote by $\rho(j)$ the subset of $\{1,2, \ldots, n\}$ that contains $j$.

For each $(a_1, \ldots, a_n) \in \X_n$, we define a pair of partitions, say $\rho_1(a_1,\ldots, a_n)$ and $\rho_2(a_1,\ldots, a_n)$, on the set $\{1,2, \ldots, n\}$ in the following way: \begin{equation}
\begin{aligned}
& i \in {\rho_1(a_1,\ldots,a_n)} (j)  \hspace{3mm} \text{if} \hspace{2mm} a_i =a_j,\\
& i\in {\rho_2(a_1,\ldots,a_n)} (j) \hspace{3mm} \text{if}\hspace{2mm}  \I (\pi(a_1,\ldots , a_n)) (a_{i}) = \I (\pi(a_1,\ldots , a_n)) (a_{j}) 
\end{aligned}\label{rho} 
\end{equation} 
Note that $\rho_1(a_1, \ldots, a_n) $ partitions  an $n$-tuple $(a_1, \ldots, a_n)\in \X^n $  according to the differential length at each $a_i$, and $\rho_2(a_1, \ldots, a_n) $ partitions $(a_1, \ldots, a_n) $ according to whether these points are siblings under the morphism $\I(\pi(a_1,\ldots, a_n))$. We thus obtain a map:
\begin{equation}
\begin{aligned}
&\rho: \X_n \longrightarrow \Pi_n \times \Pi_n\\
&(a_1,\ldots, a_n) \mapsto \rho_1(a_1,\ldots, a_n), \rho_2(a_1,\ldots, a_n)
\end{aligned}
\end{equation} 
where $\rho_1$ and $\rho_2$ are as defined in \eqref{rho}. Let $\P_n := \rho(\X_n) \subset  \Pi_n \times \Pi_n$. 

\vspace{3mm}

\noindent{\textbf{Caution!}  $\P_n$ is only a {subset} of $\Pi_n \times \Pi_n$, not a sub-poset. We will soon define a partial order on $\P_n$, and that partial order will not be the same as the one $\P_n$ inherits by virtue of being a subset of the poset $\Pi_n \times \Pi_n$.
	
\vspace{3mm}
	
For each $\alpha \in \P_n$, let $S(\alpha) := \rho^{-1}(\alpha).$  Then $S(\alpha)$ is a locally closed subset of $\X_n$ and $$\X_n = \bigsqcup\limits_{\alpha \in \P_n} S(\alpha)$$
We write $\rho_1(\alpha):= \rho_1(a_1,\ldots,a_n)$ and $\rho_2(\alpha): = \rho_2(a_1,\ldots, a_n)$, and think of them as the "coordinates" of $\alpha$  in $\Pi_n$. 

Now recall the notations and terminology set up in Definitions \ref{simplybranchedRam} and \ref{ramdata}. To each $\alpha\in \P_n$ we associate three subsets of $\{1,2,\ldots,n\}$, as follows:\begin{enumerate}
		\item $N(\alpha)$, the "simple part of $\alpha$",  index the simply-branched ramification points of morphisms in  $\I_\circ\pi(S(\alpha)) \subset \M$ are simply-branched.  Equivalently, \begin{equation*}
		\begin{aligned}
		N(\alpha): =\{i:  \text{ both }\rho_1(\alpha)(i) \text{ and }\rho_2(\alpha)(i) \text{ are singletons} \}
		\end{aligned} 
		\end{equation*}
		\item  $R(\alpha)$, the "non-simple part of $\alpha$", index ramification points of morphisms in $\I_\circ\pi(S(\alpha)) \subset \M$ with non-simple-branching. We define $$R(\alpha)_i = \{j \in \{1,\ldots,n\}-N(\alpha): j\in \rho_1(\alpha)(i)\}$$ and let
			$$R(\alpha)  :=\{R(\alpha)_i: i \in \{1,\ldots,n\}-N(\alpha)\}. $$ For all $\phi\in \I_\circ\pi(S(\alpha)) \subset \M$, note that $\# R(\alpha)$ merely denotes the number of ramification points where $\phi$ is non-simply-branched.
		\item $F(\alpha)$, the "partition of $\alpha$ into siblings", is defined via  $$F(\alpha)_j : =\{k: k\in \rho_2(\alpha)(j)  \}$$ and letting $$F(\alpha)  :=\{F(\alpha)_j: j \in\{1,\ldots,n\}-N(\alpha)\}.$$ Note that $F(\alpha)_j$, for $j\in \{1,2,\ldots,n\}-N(\alpha)$ and for any morphism in $\I_\circ\pi(S(\alpha))$, form a partition of the set of non-simply-branched ramification points according to whether they are siblings or not.  
	\end{enumerate}Now, definition \ref{ramdata} is equivalent to the "unordered version" of the sets $N(\alpha), R(\alpha)$ and  $F(\alpha)$, where $\alpha\in \P_n$ is such that $\phi\in \I_\circ\pi(S(\alpha))$. Indeed, $F(\alpha)_j$, for each $j$, up to re-ordering, is nothing but $Ram_b(\phi)$ for some $b\in \Br(\phi)$ where the ramification points are counted as many times as their differential lengths. More precisely, \eqref{FR} gives us the following:\begin{equation}
\begin{aligned}
&F(\alpha) = \{F(\alpha)_1, \ldots, F(\alpha)_r\}, \\
&F(\alpha)_j = \bigsqcup\limits_{1\leq i\leq k_j}R(\alpha)^{i}_j ,\\
&\lvert R(\alpha)^{i}_j\rvert = e^i_j-1 \hspace{2mm}\text{and}\hspace{2mm} \lvert F(\alpha)_j \rvert = k_j .
\end{aligned} 
\end{equation}
Moreover, by the Riemann-Hurwitz formula we have $\lvert N(\alpha)\rvert = n-\sum_{i,j}(e^i_j-1)$.	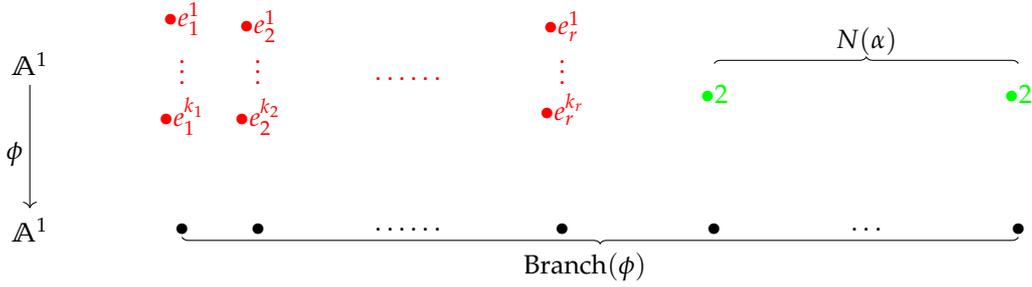
\begin{figure}[H]
	\centering
	\begin{tikzpicture}
	\node at (0,0.3) {\color{red}$\bullet$$e^1_1$}; 
	\node at (1,0.2) {\color{red}$\bullet$$e^1_2$}; 
	\node at (0,-0.3) {\color{red}$\vdots$};               
	\node at (1,-0.3) {\color{red}$\vdots$};                
	\node (A1) at (-2,-0.3) {$\A^1$};            
	\node (A2) at (-2,-2.5) {$\A^1$};            
	\draw [->] (A1)-- (A2);                         
	\node at (0,-1) {\color{red}$\bullet$$e^{k_1}_1$};                    
	\node at (-2.2,-1.5) {$\phi$};                        \node at (5,-2.5) {$\bullet$};    \node at (7,-2.5) {$\bullet$};    \node at (9,-2.5) {$\ldots$}; \node at (11,-2.5) {$\bullet$};
	\node at (1,-1) {\color{red}$\bullet$$e^{k_2}_2$};                      \node at (5,-0.3) {\color{red}$\vdots$};   \node at (7,-0.7) {\color{green}$\bullet 2$}; 
	\node at (11,-0.7) {\color{green}$\bullet 2$};
	\node at (0,-2.5) {$\bullet$};      
	\node at (3,-2.5) {$\dots\dots$};    \node at (3,-0.5) {\color{red}$\dots\dots$};                 	\node at (5,-0.9) {\color{red}$\bullet$$e^{k_r}_r$}; 
	\node at (1,-2.5) {$\bullet$};                      	\node at (5,0.2) {\color{red}$\bullet$$e^1_r$}; 
	\node at (5.3,-3) {$\Br(\phi)$};	\draw [
	decoration={
		brace,
		mirror,
		raise=0.1cm
	},
	decorate
	] (0,-2.5) -- (11,-2.5) ;
	\draw [
	decoration={
		brace,
		raise=0.5cm
	},
	decorate
	] (7,-0.8) -- (11,-0.8) ;
	\node at (9,0) {$N(\alpha)$};
%	\draw [thick, black,decorate,decoration={brace,amplitude=10pt,mirror},xshift=0.4pt,yshift=-0.4pt](0,-3) -- (7.5,-3);
	\end{tikzpicture}\caption{This is a diagrammatic presentation of the sets $N(\alpha), R(\alpha)$ and $F(\alpha)$, but only up to permutation by $\mathfrak{S}_n$. Let $\phi \in \I_\circ\pi(S(\alpha))$. The colored points are the ramification points of $\phi$, with ramification indices specified. The red points denote the non-simply-branched ramification points of $\phi$. Counted as per the differential lengths (=ramification index $-1$), they constitute $R(\alpha)$. The green points, the simply-branched ramification points of $\phi$, form $N(\alpha)$. Each "column" of red points (again, counted correctly) is an element in $F(\alpha)$.}\label{Fig3}
\end{figure}
	
\noindent	Finally,  give a partial order to $\P_n$  by reverse inclusion i.e. by declaring $$\alpha \leq \beta \iff \overline{S(\alpha)} \supseteq S(\beta).$$ 
	Put the notion of a \emph{length} on $\P_n$ given by $l: \P_n \to \{0,1,2,\ldots\}$, where \begin{gather}\label{lengthFormula}
	l(\lambda) : = \text{codim} (\overline{S_{\lambda}}) = \sum (\lvert R(\lambda)_i \rvert + 1) - \sum \lvert F(\lambda)_j \rvert - \lvert F(\lambda) \rvert.
	\end{gather} The second equality of formula \eqref{lengthFormula} follows from Proposition \ref{VBness}.
	\begin{remark}
		If $\phi \in \M$ is such that $\rho\big(\pi^{-1}(Ram(\phi))\big) =\alpha \in \P_n$, then, comparing the formulae \eqref{eq1.1} and \eqref{lengthFormula}, one obtains that $\length(\phi) = l(\alpha)$ i.e. codimension of the strata to which $\pi^{-1} \big(Ram(\phi)\big)$ belongs, equals the length of the ramification of $\phi$, as it should.
	\end{remark}
	$\P_n$ has a greatest and a least element. 	Let $\widehat{0}$ denote the  element in $\P_n$ for which $\rho^{-1}(\widehat{0})= \U_n$. Then, $N(\widehat{0}) = \{1,2,\ldots, n\}$ and $R(\widehat{1}) = F(\widehat{0}) = \emptyset$. 
	We denote by $\widehat{1}$ the element in $\P_n$ that is determined by polynomials with \emph{maximal branching}, i.e. $$\widehat{1} =\I\circ \pi\circ\rho^{-1} \{(a_1,\ldots, a_n): a_i=a_j \forall 1\leq i< j\leq n\}. $$ Then, $N(\widehat{1}) = \emptyset$ and $R(\widehat{1}) = F(\widehat{1}) = \{1,2,\ldots, n\}$. Under the partial order defined on $\P_n$, it is clear that $\widehat{0}$ and $\widehat{1}$ are the least and the greatest elements respectively.
	\subsection{$\P^m_n$ as a quotient of $\P_n$.}\label{Pmn}
	
	Fix a positive integer $m$. We construct yet another poset $\P^m_n$,  as a quotient of $\P_n$ by imposing the following equivalence relation:\begin{gather*}
	\lambda \sim \widehat{0} \iff l(\lambda) <m.
	\end{gather*} Let $\P^m_n := \P_n / \sim$, and let $$pr_m: \P_n \to \P^m_n $$ denote the corresponding quotient map. 
	The  poset $\P^m_n$  inherits a notion of length from $\P_n$, which can be defined  as follows. Let $\lambda \in \P^m_n$. Then we define the length of $\lambda$ in $\P^m_n$ via: \[l^{m}(\lambda) :=\Bigg\{\begin{array}{lr}
	0 & \text{if } l(pr_m^{-1} (\lambda))<m\\ 
	l(pr_m^{-1} (\lambda))-m & \text{if }  l(pr_m^{-1} (\lambda))\geq m
	\end{array}	\]
	$\P_n^m$ is equipped with a least and a greatest element, which we continue to denote as $\widehat{0}$ and $\widehat{1}$ by abusing notations, and where $\widehat{0}: = pr_m(\widehat{0})$ and $\widehat{1}:= pr_m(\widehat{1})$. In fact, the map $\rho:\X_n \to \P_n$ induces a map \begin{align*}
& 	\rho^{(m)}: \X_n \to \P^m_n\\ &\text{defined by } \rho^{(m)}:= pr_m\circ\rho
	\end{align*} and ${\rho^{(m)}}^{-1}(\widehat{0}) = \U^m_n$. Finally, note that $\P_n^1$ is nothing but $\P_n$ itself.
	\subsection{Action of $\mathfrak{S}_n$ on $\P_n$ and stability of the resulting quotient}\label{Section2.2}
	
	The natural action of $\mathfrak{S}_n$ on $\{1,2,\ldots, n\}$ by permutations induce an action on $\P_n$. The goal of this section is to analyse this action, and to make a precise meaning of the statement:
	
	\begin{center}
		{\emph{"The posets $\P_n/\mathfrak{S}_n$ stabilize as $n\to \infty$"}}
	\end{center}
	
There is a canonical inclusion of partially ordered sets $$\iota_n: \P_n \hookrightarrow \P_{n+1}$$ by noting that the partitions defined by \eqref{rho} on $\{1,2\ldots, n\}$, in  \Cref{Section2.1}, are compatible with those on  $\{1,2\ldots, n+1\}$. As a result, for all $\lambda \in \P_n$, we have \begin{enumerate}
		\item $R(\iota_n(\lambda)) = R(\lambda)$,
		\item $F(\iota_n(\lambda)) = F(\lambda)$, and
		\item  $N(\iota_n(\lambda)) = N(\lambda) \cup \{n+1\}$
	\end{enumerate} In particular, $l(\iota_n(\lambda)) = l(\lambda)$.
	
	There is an obvious action of the symmetric group $\mathfrak{S}_n$ on $\P_n$ induced by permutations of $\{1,2,\ldots, n\}$. Note that $\P_n/\mathfrak{S}_n$  documents only the ramification types, i.e. the data consisting of the numbers $e^i_j, k_j's$ and $r$ for various $i,j$ and $r$. It also inherits, in an obvious way, the notion of length from  $\P_n$. It follows that the following diagram of posets commutes:
	\[\begin{tikzcd}\label{CommPoset}
	\P_n \arrow[hookrightarrow]{r}{\iota_n} \arrow[twoheadrightarrow]{d}{\sigma_n} & \P_{n+1} \arrow[twoheadrightarrow]{d}{\sigma_{n+1}} \\
	\P_n/\mathfrak{S}_n  \arrow[hookrightarrow]{r}{\iota_n'} & \P_{n+1}/\mathfrak{S}_{n+1}
	\end{tikzcd}
	\]
	where $\iota_n'$ is an inclusion of partially ordered sets, and $\sigma_n$ and $\sigma_{n+1}$ denote the quotient maps by the respective symmetric groups.
	
	Clearly, $\iota_n$ is not surjective, and neither is $\iota_n'$. However, if we define $$[\P_n]_p := \{\lambda: l(\lambda) =p\}$$ then
	\begin{gather}
	\iota_n'\bigg\vert_{[\P_n]_1/\mathfrak{S}_n} :{[\P_n]_{1}/\mathfrak{S}_n} \longrightarrow {[\P_{n+1}]_{1}/\mathfrak{S}_{n+1}}\label{iota_n}
	\end{gather}
	is a bijection as long as $n\geq 3$. From a geometric perspective, this is simply because $\pi(D_{ij})$ and $\pi(T_{ij})$ are irreducible closed subvarieties of codimension $1$  in $M'_n$, for all $n\geq 3$.
	
	This begs the question: for what values of $m$, depending on $n$, is \begin{equation}
	\begin{gathered}
	\iota_n'\bigg\vert_{[\P_n]_m/\mathfrak{S}_n} :{[\P_n]_{m}/\mathfrak{S}_n} \longrightarrow {[\P_{n+1}]_{m}/\mathfrak{S}_{n+1}}
	\end{gathered}\label{stableiota}
	\end{equation} 
	a bijection? Lemma \ref{admissibile} gives an answer to this question, but before that we need make a few more definitions.
	
\begin{defn}\label{Mu-data}
Let $\mathfrak{P}:= \lim\limits_{\substack{\longrightarrow}} \P_n/\mathfrak{S}_n$, the direct limit of the system $\langle  \P_n/\mathfrak{S}_n, \iota_n'\rangle$.

\hfill$\square$
\end{defn} By the discussion above, $\mathfrak{P}$ itself inherits an obvious notion of length, which we denote by $\length:\mathfrak{P} \to \mathbb{Z}$. Thus, each $\mu \in \mathfrak{P}$ comes with the following data: \begin{align*}
& \text{positive integers } k_1,\ldots, k_r, \nonumber\\ &\text{integers } e^i_j \geq 2, \text{ for each } 1\leq i\leq k_j, \text{ and } 1\leq j \leq r, 
\end{align*} Figure \ref{Fig3} is a schematic diagram of an element in $\mathfrak{P}$. Following the notation set up in \eqref{FR}, and the formula in \eqref{lengthFormula}, if $\mu \in \mathfrak{P}$ is such that $ \mu \in \P_n/\mathfrak{S}_n$, then one has $$\length(\mu) = l(\sigma_n^{-1}(\mu)) = \sum\limits_{1\leq j \leq r} \Big(\sum\limits_{1\leq i\leq k_j}(e^i_j-1) -1\Big).$$
	We say \textbf{\textit{$\mu$ is a length $m$ ramification}} if $\length(\mu) =m$. At this juncture, one should recall Definition \ref{ramdata}. To consolidate the idea presented in definition \ref{ramdata} with what we have discussed so far, note that if $\phi\in \M$ is such that $\length(\phi) =m,$ then \begin{gather*}
	(\sigma_n\circ\rho\circ\pi^{-1}\circ\mathcal{D})(\phi) \in \mathfrak{P}  \\
	\text{and } \length(\sigma_n\circ \rho\circ\pi^{-1}\circ\mathcal{D}(\phi))=m,
	\end{gather*} which is as it should be. We define \textbf{the ramification type of $\phi$} to be $(\sigma_n\circ\rho\circ\pi^{-1}\circ\mathcal{D})(\phi) \in \mathfrak{P}$. Finally, we say  \textbf{\text{$\widetilde{\mu}$ is a  of $\mu$}} if $\widehat{\mu} \in \P_n$ and $\sigma_n(\widetilde{\mu}) =\mu$.
	
	\begin{defn}[\emph{combinatorial $n$-admissibility}]
		An element $\mu \in \mathfrak{P}$ is said to be \textbf{combinatorially $n$-admissible} if $\mu \in \P_n/\mathfrak{S}_n$. \hfill $\square$
	\end{defn} 
	
	\noindent The question posed in \eqref{stableiota} is now answered in the following lemma.
	\begin{lemma}\label{admissibile}
		For a fixed non-negative integer $m$, \textit{all} elements of $\mathfrak{P}$ having length $m$ ramification is combinatorially $n$-admissible if $n\geq 2m+1.$ 
	\end{lemma}
	\begin{remark}
		Lemma \ref{admissibile}, in other words, says that the map in \eqref{stableiota} is a bijection for $n\geq 2m+1$.
	\end{remark}
	\begin{proof}[Proof of Lemma \ref{admissibile}]
		The general principle on which the proof is based, is as follows. Let $\phi \in \M$. Let $b\in \A^1$ be a branch point of $\phi$ and let $\{t_1,\ldots, t_k\} = Ram(\phi) \cap f^{-1}(b)$ with ramification indices $m_1,\ldots, m_k$ respectively. Noting that $\sum m_j \leq n$, our goal is to find, as $\phi \in \M$ varies,  the minimum value of $n$ that would maximize $\sum\limits_{1\leq j \leq k} m_j$, keeping  $\length (\phi) = l(\rho(\pi^{-1}(\mathcal{D}(\phi)))) = m$ fixed.	
		
		Now, for any $\mu \in \mathfrak{P}$ of length $m$, we have, following the equation of length in \eqref{lengthFormula} and notations in \eqref{FR}:
		\begin{equation}
		\begin{gathered}
		m = \sum\limits_{\substack{1\leq i\leq k_j\\1\leq j \leq r}} (e^i_j-2) + \sum\limits_{1\leq j \leq r} (k_j-1)
		\end{gathered}\label{lengthp}
		\end{equation}
		Writing \eqref{lengthp} as $m = \sum_{j=1}^{r} \Big( \sum_{i=1}^{k_j}(e^i_j-1)-1\Big)$,  we first maximize $ \sum_{i=1}^{k_j}e^i_j$ for each $j$, by keeping $\sum_{i=1}^{k_j}(e^i_j-1)$ fixed. Clearly,  $ \sum_{i=1}^{k_j}e^i_j$ achieves its maximum for each $j$ when $e^i_j =2$ for all $1\leq i \leq k_j$. Therefore, plugging $e^i_j=2$ in \eqref{lengthp}, we now we have $$m =\sum\limits_{1\leq j \leq r} (k_j-1)$$ and our problem has been reduced to maximizing $k_j$ keeping $p$ fixed, for some $j$, which we can assume to be $k_1$ without any loss of generality. Clearly $k_1 = p+1$ and $k_j = 0$ for $2\leq j\leq r$ is the desired solution.  Since $n+1 \geq 2k_1$ by \eqref{FR}, we have $n+1 \geq 2(m+1)$, and so $n \geq 2m+1$. 
		
			\hfill$\square$
	\end{proof} 

\section{Poset Topology and Shellability}\label{Section3}

In this section we aim to prove some purely combinatorial results regarding our poset $\P_n$. To that end, we first recollect some generalities on posets.
\begin{defn} Let $(P,<)$ be a poset. We say that $P$ is \textbf{bounded} if it has a largest element $\widehat{1}$ and a smallest element $\widehat{0}$.
		An \textbf{$m$-chain} of $P$ is a totally ordered subset $c: = x_0< x_1< \cdots<x_m$. We say the \textbf{length of $c$} denoted by $l(c)$ is $m$.
	The \textbf{order complex} $\Delta(P)$ associated to $P$ is the simplicial complex whose $m$-simplices are the $m$-chains. 
	A chain of $P$ is \textbf{maximal} if it is inclusion-wise maximal. The elements of $\Delta(P)$ are called \textbf{faces} and the maximal faces are called \textbf{facets}. A poset is \textbf{pure} or \textbf{graded} if it is bounded and all maximal chains have the same length. 
	
	\hfill $\square$
\end{defn}

Note that for a pure poset $P$, associated to each element $\lambda \in P$ is a \textbf{\emph{length}} $l(\lambda) := l(\widehat{0},\lambda)$.

\subsection*{Cohomology of posets}\label{Section3.1}

Let $u, v \in P$ such that $u \leq v$. Let $\widetilde{C}^k(u,v)$ denote all length $k+1$ chains starting from $x_0 = u$ and ending at $x_{k+1} = v$. There are differentials \begin{equation*}
\begin{gathered}
\delta_j: \widetilde{C}^k(u,v) \to \widetilde{C}^{k+1}(u,v)\\\text{ defined by \hspace{2mm} }
\delta_j(u<x_1<\ldots< x_{k}<v)= \sum\limits_{1\leq i \leq k+1}(-1)^i (u<x_1<\ldots<\widehat{x}_i<\ldots <x_{k+1}<v).
\end{gathered}
\end{equation*} We define $H^{\ast}(u,v):= H^{\ast}( \widetilde{C}^{\bullet}(u,v))$, the cohomology of this cochain complex. 
\begin{convention}
	For $u=v$, we define $\widetilde{C}^{\bullet}(u,v)$ to consist of only $\mathbb{Z}$, placed at degree $-2$. For $u<v$, the cohomology $H^{\bullet}(u,v)$ is that of the corresponding order complex, as defined above. If $u, v \in P$ are such that there does not exist $t\in P$ for which $u <t<v$, then $H^{*}(u,v) \cong \mathbb{Z}$, placed in degree $-1$. 
\end{convention}
\begin{defn}
	Let $P$ be a pure finite poset. For $\alpha, \beta \in P$ we say that \textbf{$\alpha$ covers $\beta$} if $\alpha>\beta$ and there is no $\lambda$ such that $\alpha> \lambda>\beta$.
We say $P$ is  \textbf{semimodular} if whenever two distinct elements $\alpha, \beta \in P$ both cover $\mu \in P$ there is a $\lambda \in P$ which covers each of $\alpha$ and $\beta$. 
	$P$ is \textbf{locally semimodular} if $[\alpha,\beta]$ is semimodular for all $\alpha<\beta$ in $P$.  We say $P$ is \textbf{shellable} if the facets of $\Delta(P)$  can be arranged in linear order $F_1, F_2, . . . , F_t$ in such a way that the subcomplex $$\big(\cup_{1\leq i\leq k-1}\{G\subset F_i\}\big)\cap \{G\subset F_k\}$$ is pure and $(\text{dim} F_k -1)$-dimensional for all $k = 2, \ldots , t$.

\hfill $\square$
\end{defn}

\begin{lemma}[Theorem 6.1 of ~\cite{BjoernerJuly1980}]
	Suppose that a finite poset $P$ is bounded and locally semimodular. Then $P$ is shellable.
\end{lemma}

\begin{lemma}
	If $P$ is shellable, then for all $\lambda \in P$, we have $\widetilde{H}^i(\widehat{0},\lambda) =0$ whenever $i< l(\lambda) -2$.
\end{lemma}

\begin{proof}
	For a proof, see Section 4.1 of ~\cite{Wachs2006}. 
	
		\hfill$\square$
\end{proof}

\noindent The next proposition is the key takeaway from this section, and forms the second crucial step in our proof of Theorem \ref{thmB} (see the proof outline on page 4). Recall the posets $\P^m_n$ defined in Section \ref{Pmn}.
\begin{proposition}\label{LocalSemiMod}
	 Let $m$ and $n$ be positive integers. Then $\P^m_n$ is locally semimodular for all $m$ and $n$ that satisfy $m\leq n$.
\end{proposition}
\begin{proof}
	For simplicity, we prove the statement for $m=1$ i.e. for $\P_n$. The exact argument works for  $m\geq 2$ since every interval in $P^m_n$ is actually an interval in $P_n$.
	
	So now, our goal is to show that $\P_n$ is locally semimodular.
	The statement is trivial for $n<2$. So, we assume $n\geq 2$.
	Let $[x,y]$ be an interval in $\P_n$. To prove the proposition we can safely assume $l([x,y]) \geq 2$, since otherwise, the statement is vacuously true. It suffices to show that if $u$ and $v$ cover $x$ then there exists $t \in \P_n$ such that $t \leq y$ and $t$ covers both $u$ and $v$.
	
	If $u$ and $v$ cover $x$ then $$\overline{S(x)} \supset S(u) \cup S(v)$$ Also, let $l(x) = m$, so $l(u) = l(v) = m+1$ since $\P_n$ is a graded pure poset. Consider a maximal chain in $[\widehat{0},x]$. Suppose $$\overline{S(x)}\subset Z_1 \cap \cdots \cap Z_m$$ i.e. $\overline{S(x)}$ is an irreducible component of $Z_1 \cap \cdots \cap Z_m$, where, for each $k$, we have $Z_k  =D_{ij}$ or $Z_k = T_{ij}$ for some $i,j$ .  
	Since $u \neq v$ there exist two distinct divisors, let's call them $Z_{m+1}$ and $Z'_{m+1}$ such that \begin{gather*}
	\overline{S(x)} \cap Z_{m+1} \supset S(u), \hspace{5mm}
	\overline{S(x)} \cap Z'_{m+1} \supset S(v),\\
	S(v) \subsetneq\overline{S(x)} \cap Z_{m+1}\hspace{3mm}\text{and}\hspace{3mm}
	S(u) \subsetneq \overline{S(x)} \cap Z'_{m+1}.
	\end{gather*}  This forces $\bigcap\limits_{1\leq k \leq p}Z_k \cap Z_{m+1} \cap Z'_{m+1}$ to have codimension $m+2$, and to have a component whose generic point gives rise to an element $\P_n$, say $t$, such that $t$ covers $u$ and $v$ and such that $\overline{S(t)} \supseteq S(y)$.
	
	\hfill$\square$
\end{proof}

\begin{remark}
	The intersections of $D_{ij}$ and $T_{ij}$ for various values of $i$ and $j$ are not always irreducible. In combinatorial language, one  says "$\P_n$ doesn't admit  meets, and  joins." 
\end{remark}
\begin{cor}\label{semimodular}
	Let $n$ and $m$ be positive integers that satisfy $m<n$. Then for all $\lambda \in \P^m_n$ we have $\widetilde{H}^i(\widehat{0},\lambda) =0$ whenever $i< l^{m}(\lambda) -2$.
\end{cor}
\begin{proof}
	Use Proposition ~\ref{LocalSemiMod}, Lemma 4.4 and Lemma 4.5, in that order. \hfill$\square$
\end{proof}
\section{Geometry of the (Zariski) closure of the locally closed strata}\label{Section4}
The components of $\X_n - \U^m_n$, for  $m<n$, are quite far from "nice": they are singular, they don't intersect transversally, etc. For example, when $n> 3$ the homogeneous equations cutting out the divisors $D_{ij}$ are of degree $n-2$, and thus have no linear part.  Therefore, the divisors $D_{ij}$ are not smooth at the origin. But that is not too much of a problem-- the closed strata in $\X_n$, given by  ramification types, have quotient singularities when their codimension $\ll n$, which make them quite tractable.
 The purpose of this section is twofold: given $m\geq 1$ and $n\geq 3m$; and $\lambda\in \P_n$ such that $l(\lambda) = m$, \begin{enumerate}
	\item check that $S(\lambda)$ is non-empty, and
	\item prove that the quotient $X(\lambda)/\mathfrak{S}_{\lvert N(\lambda)\rvert}$ is isomorphic to an affine space.
\end{enumerate}
We address the second problem first. 

\begin{proposition}\label{VBness}
	For $\lambda \in \P_n$, let $X(\lambda) := \overline{S(\lambda)}$ and let $N_0:= n- \sum_{i}\lvert R(\lambda)_i\rvert + \lvert F(\lambda)\rvert $. Then $$X(\lambda)/\mathfrak{S}_{\lvert N(\lambda) \rvert}\xrightarrow{\cong} \mathbb{A}^{\lvert R(\lambda) \rvert+N_0}$$ whenever $N_0\geq 0$.
\end{proposition}
\begin{proof}
	What we will actually show is $$X(\lambda)/\mathfrak{S}_{\lvert N(\lambda) \rvert} \longrightarrow \mathbb{A}^{\lvert R(\lambda) \rvert}$$ is an affine space bundle with fibres isomorphic to $\A^{N_0}$. Once we prove this, the statement of the proposition is then a direct consequence of the Quillen-Suslin theorem (a.k.a Serre's conjecture, see e.g., \cite[ Theorem 3.7, Chapter XXI]{Lang2002}) which states that finite projective modules over polynomial rings over a field are free. 
	For the sake of simplicity, we consider three cases; the first two will just turn out to be special cases of the third one. 

	\noindent \textbf{Case 1:}
	We prove the proposition for those $\lambda \in \P_n$ for which the polynomials in $S(\lambda)$ have no more than one ramification point in each fibre i.e. $R(\lambda) = F(\lambda)$. We continue with the notation from \eqref{FR}, except, for convenience, we write $e_j:= e^i_j$ since $i=1$ for each $1\leq j\leq r$.  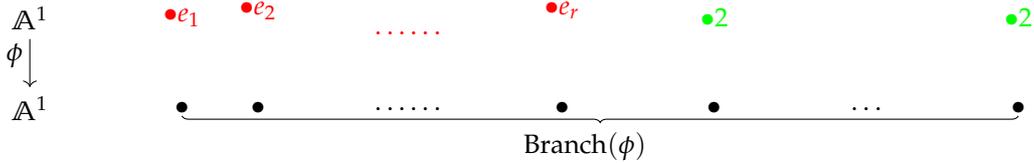
\begin{figure}[H]
		\centering
		\begin{tikzpicture}
		\node at (0,-1.3) {\color{red}$\bullet$$e_1$}; 
		\node at (1,-1.2) {\color{red}$\bullet$$e_2$};             
		\node (A1) at (-2,-1.3) {$\A^1$};            
		\node (A2) at (-2,-2.5) {$\A^1$};            
		\draw [->] (A1)-- (A2);                         
	                   
		\node at (-2.2,-1.8) {$\phi$};                        \node at (5,-2.5) {$\bullet$};    \node at (7,-2.5) {$\bullet$};    \node at (9,-2.5) {$\ldots$}; \node at (11,-2.5) {$\bullet$};
	                     \node at (7,-1.3) {\color{green}$\bullet 2$}; 
		\node at (11,-1.3) {\color{green}$\bullet 2$};
		\node at (0,-2.5) {$\bullet$};      
		\node at (3,-2.5) {$\dots\dots$};    \node at (3,-1.5) {\color{red}$\dots\dots$};                 
		\node at (1,-2.5) {$\bullet$};                      	\node at (5,-1.2) {\color{red}$\bullet$$e_r$}; 	\node at (5.3,-3) {$\Br(\phi)$};	\draw [
		decoration={
			brace,
			mirror,
			raise=0.1cm
		},
		decorate
		] (0,-2.5) -- (11,-2.5) ;
		\end{tikzpicture}\caption{The above diagram is an example of a morphism in $\I_\circ\pi(S(\lambda))$ satisfying $R(\lambda) = F(\lambda)$ or equivalently, $B_b(\phi)$ is a singleton for all $b\in\Br(\phi)$.}\label{Fig4}
	\end{figure}	Define \begin{equation}
	\begin{gathered}
	\begin{split}
	\mathcal{Z}(\lambda) := \bigg\{ ((a_{1}, \ldots, a_{r}), f): f \in M'_n,\, f(x) = (x-a_{1})^{e_{1} - 1} \ldots (x-a_{r})^{e_{r} - 1} g(x), \\ g(x) \hspace{2mm}\text{monic of degree} \hspace{2mm}  n- \sum_{i=1}^{r}(e_j -1) \bigg\}
	\end{split}
	\end{gathered}
	\end{equation} 
	First, note that there is a natural surjective morphism $$X(\lambda) \twoheadrightarrow X(\lambda)/\mathfrak{S}_{\lvert N(\lambda)\rvert}$$ is given by keeping the coordinates indexed by $\{1,\ldots,n\} -N(\lambda)$  fixed, while the coordinates indexed by $N(\lambda)$ map to the corresponding elementary symmetric polynomials in $\lvert N(\lambda)\rvert$ variables. The coordinates indexed by $\{1,\ldots,n\} -N(\lambda)$ has repetitions, indexed precisely by $R(\lambda)$. Forgetting the repetitions show that \begin{equation}
	\begin{gathered}
	\Phi: X(\lambda)/\mathfrak{S}_{\lvert N(\lambda)\rvert} \xrightarrow{\cong} \Z(\lambda)
	\end{gathered}\label{isom7}
	\end{equation}
	Now, let $N : = \lvert N(\lambda) \rvert = n- \sum_{i=1}^{r}(e_j -1)$, and define a morphism \begin{equation}
	\begin{gathered}
	\Psi: \A^r \times A^{N} \longrightarrow \Z(\lambda)\\
\Big((a_1,\ldots, a_r), (s_1,\ldots, s_{N})\Big) \mapsto\Bigg( (a_1,\ldots, a_r), \Big((x-a_{1})^{e_{1} - 1} \ldots (x-a_{r})^{e_{r} - 1} (x^N+ s_1x^{N-1}+ \ldots+ s_N)\Big)\Bigg)
	\end{gathered}\label{isom8}
	\end{equation} is clearly an isomorphism.
	In conclusion, $$\Psi^{-1}\circ \Phi: X(\lambda)/\mathfrak{S}_{\lvert N(\lambda)\rvert} \to \A^r \times A^{N} $$ is an isomorphism, and if 
	$$\varphi: \Z(\lambda) \longrightarrow \A^r$$ denotes the projection to the first $r$ coordinates, then $\Z(\lambda)$ is a trivial $\A^N$-bundle over $A^r$, thus completing the proof of  Proposition \ref{VBness} for Case 1.
	
	\vspace{4mm}
	
	\noindent\textbf{Case 2:} 	Let $\lambda$ be such that $F(\lambda)$ is a singleton. So, following the notation from \eqref{FR}, we have $r=1$. Letting $k: = k_1$, and $e^i = e^i_j$ since $j$ can only be $1$, we have $F(\lambda) = \bigsqcup\limits_{1\leq i\leq k} R(\lambda)_i$. 	\begin{figure}[H]
		\centering
		\begin{tikzpicture}
		\node at (4,0.8) {\color{red}$\bullet$$e^1$}; 
		\node at (4,0) {\color{red}$\bullet$$e^2$}; 
		%\node at (1,0.2) {\color{red}$\bullet$$e^2$}; 
		\node at (4,-0.5) {\color{red}$\vdots$};               
	%	\node at (1,-0.3) {\color{red}$\vdots$};                
		\node (A1) at (1.5,-0.3) {$\A^1$};            
		\node (A2) at (1.5,-2.5) {$\A^1$};            
		\draw [->] (A1)-- (A2);                         
		\node at (4,-1) {\color{red}$\bullet$$e^{k}$};                    
		\node at (1.3,-1.5) {$\phi$};                          \node at (7,-2.5) {$\bullet$}; \node at (9,-2.5) {$\ldots$};    \node at (11,-2.5) {$\bullet$};
	  \node at (7,-0.7) {\color{green}$\bullet 2$};  \node at (9,-0.7) {\color{green}$\ldots$}; 
		\node at (11,-0.7) {\color{green}$\bullet 2$};
		\node at (4,-2.5) {$\bullet $};             
		\node at (7.2,-3) {$\Br(\phi)$};	\draw [
		decoration={
			brace,
			mirror,
			raise=0.1cm
		},
		decorate
		] (4,-2.5) -- (11,-2.5) ;
		\end{tikzpicture}\caption{The above diagram characterises $\lambda\in \P_n$ such that for any morphism $\phi\in \I_\circ\pi(S(\lambda))$, we have that $B_b(\phi)=\{2\}$ for all branch points $b\in \Br(\phi)$  but one.}
	\end{figure}
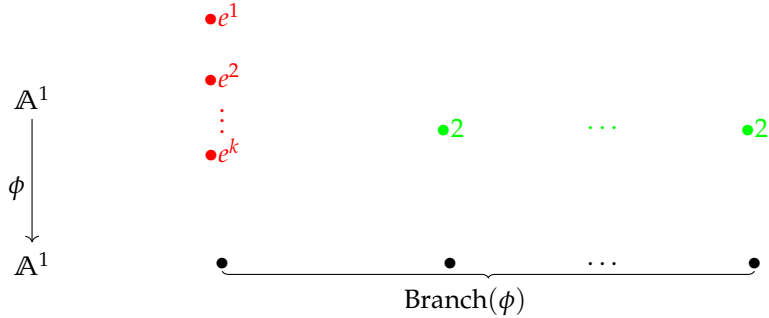 \noindent Define \begin{equation}
	\begin{aligned}\label{Z2}
	\Z(\lambda) := \bigg\{ \big((a_{1}, \ldots, a_{k}, f\big):  f \in M'_n,\, f(x) = (x-a_{1})^{e^{1} - 1} \ldots (x-a_{k})^{e^{k} - 1} g(x), \\ \I(f) (a_{1}) = \I(f) (a_{j}), j = 2, \ldots, k  \\ g(x) \hspace{2mm}\text{monic of degree} \hspace{2mm}  n- \sum\limits_{1\leq i \leq k}(e^i -1)\bigg\}
	\end{aligned} 
	\end{equation} and let $$\varphi: \Z(\lambda) \to \A^k$$ denote the projection to the first $k$ coordinates.
	
	The proof of ~\eqref{isom7} from Case 1 carries over verbatim to Case 2, and we have an isomorphism: $$\Phi: \Z(\lambda) \overset{\cong}\longrightarrow X(\lambda)/\mathfrak{S}_{\lvert N(\lambda)\rvert}$$
	We have only to show that $\varphi: \Z(\lambda) \to \mathbb{A}^{k}$,  is a fibre bundle with fibres isomorphic to  $\A^{n- \sum_i(e^i -1)-(k-1)}$. 
	This fact was obvious in Case 1, but requires some extra work for Case 2, which we explain now.
	As in the proof of Case 1, we have an affine space bundle over $\mathbb{A}^k$ defined by  \begin{align*}
	\E(\lambda) := \bigg\{ \big(a_{1}, \ldots, a_{k}, f\big):  f \in M'_n,\, f(x) = (x-a_{1})^{e^{1} - 1} \ldots (x-a_{k})^{e^{k} - 1} g(x), \\ g(x) \hspace{2mm}\text{monic of degree} \hspace{2mm}  n- \sum\limits_{1\leq i \leq k}(e^i -1) \bigg\}.
	\end{align*}  Let $$\widetilde{\varphi}: \E(\lambda) \to \A^k$$denote the projection to the first $k$ coordinates. Clearly, the fibres of $\widetilde{\varphi}$ are spanned by the coefficients of $g(x)$, and we have
	$$\widetilde{\varphi}^{-1}\Big((a_1,\ldots,a_k)\Big)  \cong \A^{n-\sum\limits_{1\leq i \leq k}(e^i-1)}$$ 
	Similar to  ~\eqref{isom8}, if $N: =n- \sum\limits_{1\leq i \leq k}(e^i -1)$ we have an isomorphism  \begin{equation}
	\begin{gathered}
	\widetilde{\Psi}: \A^k \times A^{N} \longrightarrow \E(\lambda)\\
\Big((a_1,\ldots, a_k), (s_1,\ldots, s_{N})\Big) \mapsto \Bigg((a_1,\ldots, a_k), \Big((x-a_{1})^{e^{1} - 1} \ldots (x-a_{k})^{e^{k} - 1} (x^N+ s_1x^{N-1}+ \ldots+ s_N)\Big)\Bigg)
	\end{gathered}\label{isom9}
	\end{equation}  and the following diagram commutes \[\begin{tikzcd}
	\Z(\lambda) \arrow[hookrightarrow]{r} \arrow[twoheadrightarrow]{dr}{\varphi} & \E(\lambda) \arrow[twoheadrightarrow]{d}{\widetilde{\varphi}} \\
	& \A^{k}
	\end{tikzcd}
	\]  Since for any $(a_1,\ldots, a_k) \in A^k$, the affine space $\varphi^{-1}(a_1,\ldots, a_k)$ is a linear subspace of $\widetilde{\varphi}^{-1}(a_1,\ldots, a_k)$, to prove Proposition \ref{VBness}  for Case 2, it suffices to show that the fibres of $\varphi$ have constant dimension.
	
	\noindent To this end, write $$\I(f)(x)-c = (x-a_{1})^{e^{1}}\ldots (x-a_{k})^{e^{k}} h(x)$$ for some $c \in \mathbb{A}^1$ and some monic polynomial $h(x)$ of degree $n+1-\sum e^i$. Then, taking derivatives, we obtain: \begin{align*}
	f(x) = (x-a_1)^{e^1-1} \ldots (x-a_k)^{e^k-1} \Big((x-a_1)\ldots(x-a_k)h'(x) +\\ e^1(\widehat{x-a_1})(x-a_2)\ldots(x-a_k)h(x) + \\ e^2(x-a_1) (\widehat{x-a_2})\ldots(x-a_k)h(x) + \\e^k(x-a_1)\ldots (x-a_{k-1}) (\widehat{x-a_k}) \Big)
	\end{align*}
	where $(\widehat{x-a_j})$ signifies that that factor is removed. Comparing with the  expression for $f(x)$ in \eqref{Z2} we obtain: \begin{align*}
	g(x) = (x-a_1)\ldots(x-a_k)h'(x) + e^1(\widehat{x-a_1})(x-a_2)\ldots(x-a_k)h(x) + \\ e^2(x-a_1) (\widehat{x-a_2})\ldots(x-a_k)h(x) + \\e^k(x-a_1)\ldots (x-a_{k-1}) (\widehat{x-a_k}) 
	\end{align*} and we see that the coefficients of $h(x)$ span a linear subspace, of dimension ${n+1-\sum e^i}$, of the affine space generated by the coefficients of $g(x)$. Therefore, for any $(a_1,\ldots,a_k)\in\A^k$ we have that  $$\varphi^{-1}\Big((a_1,\ldots,a_k)\Big)\cong \A^{n+1-\sum e^i}.$$ This completes the proof of  Proposition \ref{VBness} for Case 2.
	
	\vspace{4mm}
	
	\noindent\textbf{Case 3:} Finally, the general case, as depicted in Figure \ref{Fig3}. The proof resembles that of Case 2 very closely, but we nevertheless try to be as explicit  possible for the sake of clarity.
	As before, let $\lambda\in \P_n$. We follow the notations set in \eqref{FR}, \Cref{Section2.2}, which we recollect here for convenience. Let \begin{enumerate}
		\item $F(\lambda) = \{F(\lambda)_1, \ldots, F(\lambda)_r\}$, so $\lvert F(\lambda)\rvert =r$
		\item  $F(\lambda)_j = \bigsqcup\limits_{1\leq i\leq k_j}R(\lambda)^{i}_j$ 
		\item $\lvert R(\lambda)^{i}_j\rvert = e^i_j-1$ and $\lvert F(\lambda)_j \rvert= k_j $
	\end{enumerate}  In other words, if $(a_1,\ldots, a_n)\in \rho^{-1}(\lambda) \subset \X_n$, then by definition \ref{ramdata}, $\I(\pi(a_1,\ldots,a_n))$ is a polynomial satisfying the following: for all branch points $b\in\I(\pi(a_1,\ldots,a_n))$ that are not simple, $Ram_b(\I(\pi(a_1,\ldots,a_n)))\in \Sym^{l(B_b(\I(\pi(a_1,\ldots,a_n))))}\A^1$ is given by (see \eqref{FR}): \begin{align*}
Ram_b(\I(\pi(a_1,\ldots,a_n))) = \Big(\underbrace{a^{1},\ldots, a^{1}}_{e^1-1}, \ldots,\underbrace{a^i,\ldots, a^i}_{e^i-1}\Big). 
\end{align*}where $\{a^1,\ldots,a^i\} \in Ram(\I(\pi(a_1,\ldots,a_n)))$ and $B_b(\I(\pi(a_1,\ldots,a_n))) = \{e^1,\ldots, e^k\}.$
 Define $\Z(\lambda)$ the same way as in Case 2, namely \begin{equation}
	\begin{aligned}
	\Z(\lambda) := \bigg\{ \big((a^1_{1}, \ldots, a^{k_1}_1), \ldots, (a^1_j,\ldots, a^{k_j}_j), \ldots, (a^1_r, \ldots, a^{k_r}_r), f\big):  f \in M'_n,\\ f(x) = \prod\limits_{1\leq i\leq k_1}(x-a^i_{1})^{e^i_{1} - 1} \ldots \prod\limits_{1\leq i\leq k_r}(x-a^i_{r})^{e^i_{r} - 1} g(x), \\ \I(f) (a^i_{j}) = \I(f) (a^1_{j}), \,\, 2\leq i\leq k_j, \,\,1\leq j\leq r  ,\\ g(x) \hspace{2mm}\text{monic of degree} \hspace{2mm}  n- \sum\limits_{\substack{1\leq i\leq k_j\\1\leq j \leq r}}(e^i_j -1)\bigg\}
	\end{aligned}\label{Z3} 
	\end{equation}and  let $$\varphi: \Z(\lambda) \to \A^{\sum k_j}$$ denote the projection to the first $\sum k_j$ coordinates. Similarly, define \begin{equation}
	\begin{aligned}
	\E(\lambda) :=  \bigg\{ \big((a^1_{1}, \ldots, a^{k_1}_1), \ldots, (a^1_j,\ldots, a^{k_j}_j), \ldots, (a^1_r, \ldots, a^{k_r}_r), f\big):  f \in M'_n,\\ f(x) = \prod\limits_{1\leq i\leq k_1}(x-a^i_{1})^{e^i_{1} - 1} \ldots \prod\limits_{1\leq i\leq k_r}(x-a^i_{r})^{e^i_{r} - 1} g(x) ,\\ g(x) \hspace{2mm}\text{monic of degree} \hspace{2mm}  n- \sum\limits_{\substack{1\leq i\leq k_j\\1\leq j \leq r}}(e^i_j -1)\bigg\}
	\end{aligned}\label{E3}
	\end{equation}and let $$\widetilde{\varphi}: \E(\lambda) \to \A^{\sum k_j}$$denote the projection to the first $\sum k_j$ coordinates. Clearly, the fibres of $\widetilde{\varphi}$ are generated by the coefficients of $g(x)$, and we have
	$$\widetilde{\varphi}^{-1}\Big((a_1,\ldots,a_k)\Big)  \cong \A^{n-\sum\limits_{1\leq i \leq k}(e^i-1)}$$ 
	In fact, the isomorphism in \eqref{isom9} carries over verbatim, just with $k$ replaced by $\sum k_j$. 
	As in Case 2, we now have the following commutative diagram \[\begin{tikzcd}
	\Z(\lambda) \arrow[hookrightarrow]{r} \arrow[twoheadrightarrow]{dr}{\varphi} & \E(\lambda) \arrow[twoheadrightarrow]{d}{\widetilde{\varphi}} \\
	& \A^{\sum k_j}
	\end{tikzcd}
	\] and our goal is to show that the fibres of $\varphi$ have constant dimension.
	For each $j$, we can write $$\I(f) (x)-c_j = (x-a^1_{j})^{e_j^{1}}\ldots (x-a^{k_j}_{j})^{e^{k_j}_{j}} h_j(x)$$ for some $c_j \in \mathbb{C}$ and some monic polynomial $h_j(x)$ of degree $n+1 -\sum\limits_{1\leq i\leq k_j}e^{i}_j$. 
	Therefore: \begin{align*}
	f(x) = (x-a^1_j)^{e^1_j-1} \ldots (x-a^{k_j}_j)^{e^{k_j}_j-1} \Big((x-a^1_j)\ldots(x-a^{k_j}_j)h'_j(x) +\\ e^1_j(\widehat{x-a^1_j})(x-a^2_j)\ldots(x-a^{k_j}_j)h_j(x) + \\ e^2_j(x-a^1_j) (\widehat{x-a^2_j})\ldots(x-a^{k_j}_j)h_j(x) + \\e^{k_j}_j(x-a^1_j)\ldots (x-a^{k_j-1}_j) (\widehat{x-a^{k_j}_j}) \Big)\\ \text{for each} \, \, 1\leq j\leq r
	\end{align*}	
	Comparing with original expression for $f(x)$ in \eqref{E3}, we see that for each $j$: \begin{align*}
	g(x) = (x-a^1_j)\ldots(x-a^{k_j}_j)h'_j(x) + e_1(\widehat{x-a^1_j})(x-a^2_j)\ldots(x-a^{k_j}_j)h_j(x) + \\ e^2_j(x-a^1_j) (\widehat{x-a^2_j})\ldots(x-a^{k_j}_j)h_j(x) + \\e_k(x-a^1_j)\ldots (x-a^{k_j-1}_j) (\widehat{x-a^{k_j}_j}) 
	\end{align*} 
	For any $\big((a^1_{1}, \ldots, a^{k_1}_1), \ldots, (a^1_j,\ldots, a^{k_j}_j), \ldots, (a^1_r, \ldots, a^{k_r}_r)\big) \in \A^{\sum k_j}$, much like the proof of Case 2, we have \begin{gather*}
	\varphi^{-1}\big((a^1_{1}, \ldots, a^{k_1}_1), \ldots, (a^1_j,\ldots, a^{k_j}_j), \ldots, (a^1_r, \ldots, a^{k_r}_r)\big) = \\ \bigcap\limits_{j}\Big\{\text{linear subspace of} \hspace{2mm}\widetilde{\varphi}^{-1}\big((a^1_{1}, \ldots, a^{k_1}_1), \ldots, (a^1_j,\ldots, a^{k_j}_j), \ldots, (a^1_r, \ldots, a^{k_r}_r)\big) \\ \text{spanned by the coefficients of } h_j\Big\}
	\end{gather*}
	Let \begin{gather*}
	V_j \big((a^1_{1}, \ldots, a^{k_1}_1), \ldots, (a^1_j,\ldots, a^{k_j}_j), \ldots, (a^1_r, \ldots, a^{k_r}_r)\big) := \\ \Big\{\text{linear subspace of} \hspace{2mm}\widetilde{\varphi}^{-1}\big((a^1_{1}, \ldots, a^{k_1}_1), \ldots, (a^1_j,\ldots, a^{k_j}_j), \ldots, (a^1_r, \ldots, a^{k_r}_r)\big) \\ \text{spanned by the coefficients of } h_j\Big\}
	\end{gather*} Then, noting that the degree of $h_j$ is $n+1 -\sum\limits_{1\leq i\leq k_j}e^{i}_j$, we have that the codimension of $V_j$ in $\widetilde{\varphi}^{-1}\big((a^1_{1}, \ldots, a^{k_1}_1), \ldots, (a^1_j,\ldots, a^{k_j}_j), \ldots, (a^1_r, \ldots, a^{k_r}_r)\big) $ is $\sum\limits_{1\leq i\leq k_j}e^{i}_j-1$. Therefore, \begin{equation}
	\begin{gathered}
	\text{codimension of }\Big( \bigcap\limits_{j} V_j\big((a^1_{1}, \ldots, a^{k_1}_1), \ldots, (a^1_j,\ldots, a^{k_j}_j), \ldots, (a^1_r, \ldots, a^{k_r}_r)\big) \Big) \\ \leq \sum\limits_{i,j} e^i_j -r
	\end{gathered} \label{ineq6}
	\end{equation}where equality holds if the intersection of these linear subspaces is (dimensionally) transverse.
	What is left to show is that the inequality in ~\eqref{ineq6} is actually an equality  over all points in $\A^{\sum k_j}$. 
	
	To this end, note that when all the ramification points come together, i.e. when $a_i =a_j$ for all $i,j$, we are reduced to Case 1. In that situation,~\eqref{ineq6} reduces to an equality $$\text{dim}\Big(\varphi^{-1} (\underbrace{a,\ldots \ldots \ldots,a}_\mathrm{\sum\limits_{i,j} e^i_j -r})\Big) = n+r-\sum\limits_{i,j} e^i_j .$$ Now, upper-semicontinuity of the dimension of the fibres (see e.g. \cite[Corollaire 13.1.5]{Grothendieck1966}) implies that for all $$\Big((a^1_{1}, \ldots, a^{k_1}_1), \ldots, (a^1_j,\ldots, a^{k_j}_j)\Big) \in \A^{\sum k_j}$$ one has   \begin{equation}
	\begin{gathered}
	\text{codimension of }\Big( \bigcap\limits_{j} V_j\big((a^1_{1}, \ldots, a^{k_1}_1), \ldots, (a^1_j,\ldots, a^{k_j}_j), \ldots, (a^1_r, \ldots, a^{k_r}_r)\big) \Big) \\ \geq \text{codimension of }\Big(\varphi^{-1} (\underbrace{a,\ldots \ldots \ldots,a}_\mathrm{\sum\limits_{i,j} e^i_j -r})\Big) \\= \sum\limits_{i,j} e^i_j -r
	\end{gathered}.\label{UpSemIneq} 
	\end{equation}
	
	Finally, note that ~\eqref{ineq6} and ~\eqref{UpSemIneq} now imply:
	\begin{gather*}
	dim\Bigg(\varphi^{-1}\big((a^1_{1}, \ldots, a^{k_1}_1), \ldots, (a^1_j,\ldots, a^{k_j}_j), \ldots, (a^1_r, \ldots, a^{k_r}_r)\big)\Bigg) = n+r-\sum\limits_{i,j} e^i_j \\=n+ \lvert F(\lambda)\rvert - \sum_{i}\lvert R(\lambda)_i\rvert 
	\end{gather*}
	which completes the proof. 
	
	\hfill $\square$
	
\end{proof}

\begin{remark}[\textbf{Irreducibility of $X(\lambda)$}]
	If $\lambda \in \P_n$ is such that $n +r -\sum e^i_j\geq 0$, then Proposition \ref{VBness} implies that if $S(\lambda)$ is non-empty then it admits a finite, unramified morphism to $S(\lambda)/\mathfrak{S}_{\lvert N(\lambda)\rvert}$, given by:  $$\varphi \Big\vert_{S(\lambda)}: S(\lambda) \to  S(\lambda)/\mathfrak{S}_{\lvert N(\lambda)\rvert}.$$ The deck group $\mathfrak{S}_{\lvert N(\lambda)\rvert}$  acts transitively on the fibres of $\varphi \Big\vert_{S(\lambda)}$. So $S(\lambda)$ is connected, and  its closure $X(\lambda)$ is irreducible.
\end{remark}

Now we address the first question raised in the beginning of \Cref{Section4}, namely, the question of non-emptyness of $S(\lambda)$, for  $\lambda \in \P_n$. 

\begin{claim}\label{claim4.3}
	Given $e_1,\ldots, e_r$ such that $e_j \geq 2$ and $\sum (e_j -1) = n$, there exists $\phi \in \M$ such that $B_b(\phi)$ is a singleton for all $b\in \Br(\phi)$.
\end{claim}
 In other words, we are proving the non-emptyness of the strata corresponding to those elements of $\P_n$ that are handled in Case 1 of Proposition \ref{VBness} (see Figure \ref{Fig4}.). For the definitions of $B_b(\phi)$ and $\Br(\phi)$ see definition \ref{ramdata} and \eqref{FR}.

\vspace{3mm}

\begin{proof}{(of Claim \ref{claim4.3})} We divide the proof into two cases: when $char \K =0$ and and when $char\K>0$.
	
	First, we consider the case when $char \K=0$. Our strategy is to prove the statement for when $K=\mathbb{C}$ and then invoke a theorem by Grothendieck to prove the statement for a general field of characteristic $0$.  So now, assume $\K =\mathbb{C}$. Suppose we are given $B: =\{b_1,\ldots, b_r\}\subset \A^1$, and $r$ simple disjoint oriented paths $\gamma_1, \ldots, \gamma_r$ starting at a given base point, say $b_0$, and encircling $b_1, \ldots, b_r$ respectively. In particular, $\gamma_1, \ldots, \gamma_r$ freely generates $\pi_1^{top}(\A^1-B, b_0)$, the topological fundamental group of $\A^1-B.$
	On the other hand, let $\tau_1,\ldots, \tau_r$ denote cycles in $\mathfrak{S}_{n+1}$ of length $e_1, \ldots, e_r$ respectively, such that their product is an $n+1$-cycle. For example, one can choose \begin{gather*}
	\tau_1 = (1 \ldots e_1),\\ \tau_2 = (e_1 \ldots (e_1+e_2-1)),\\ \tau_3 = ((e_1+e_2-1) \ldots (e_1+e_2+e_3-2)),\\ \vdots\\ \tau_k =\Big(\big(\sum_{i=1}^{k-1}e_i-(k-2)\big) \ldots \big(\sum_{i=1}^{k} e_i-(k-1)\big)\Big), \\ \vdots
	\end{gather*}Clearly $\prod \tau_i = (1\ldots n+1)$.
	%any two cycles overlap on at most one element of $\{1,2\ldots, n+1\}$. 
	Now consider the homomorphism \begin{gather}
	\pi_1^{top}(\A^1-B, b_0) \to \mathfrak{S}_{n+1} \nonumber\\
	\gamma_i \mapsto \tau_i \label{fundemantalHom}
	\end{gather} This induces an algebraic cover $ \phi: \A^1 \to \mathbb{A}^1$ by Riemann's existence theorem, and hence is given by a polynomial of degree $n+1$. More explicitly, we can take the disjoint union of $n+1$ copies of $\A^1 -\cup_j \{\text{interiors of } \gamma_j\}$ and 'join them locally' over the disks bounded by $\gamma_j$ by $z \mapsto z^{e_j}$ for $1\leq j\leq r$. A beautiful explanation along these lines can be found in  ~\cite{Eisenbud1991}.
	Forgetting the choice of a base point entails defining the epimorphisms up to conjugacy, and in turn we have  the following bijective correspondence: \begin{gather*}
	\text{Hom}\big(\pi_1(\A^1-B, b_0), \mathfrak{S}_{n+1}\big)/\text{Inn}(\mathfrak{S}_{n+1}) \\\longupdownarrow \\ {\Big\{ \text{degree n+1 monic polynomials branched at }b_1,\ldots, b_r \text{with ramification indices }e_1,\ldots, e_r \text{ respectively}\Big\}}.
	\end{gather*}Over a general algebraically closed field of characteristic $0$, the \etale fundamental group of the "$\A^1-B$" is isomorphic to that over $\mathbb{C}$, as proved by Grothendieck in \cite[Proposition 4.6 (Formule de Künneth)]{Grothendieck1971}. So the above argument carries over verbatim with $\pi_1^{top}$ replaced by $\pi_1^{\et}$. 
	
	Next, consider the case when $char \K =p>0$. Recall that at the beginning of this paper, we fixed once and for all, that whenever $char \K>0$, we assume $char \K>n+1$. As a result, $p \not\big\vert (n+1)!$ and branched $\mathfrak{S}_{n+1}$ covers of $\mathbb{P}^1_{\K}$ are in bijective correspondence with branched covers of $\mathbb{P}^1_{\mathbb{C}}$ (see, e.g. \cite{Stevenson2017}). 
	More precisely, as  explained in \textit{loc. cit.}, one considers \etale covers of degree $n+1$ over $\A^1-\{r \text{ points}\}$ as finite quotients of the \textit{prime-to-}$p$ fundamental group $\pi_1^{p'}$. It is defined by taking the inverse system of \etale covers, the order of whose Galois group is coprime to $p$. Note that $\pi_1^{p'}$ is the maximal prime-to-$p$ quotient of $\pi_1^{\et}$, which itself is defined by considering the inverse system of all finite \etale covers of $\A^1-\{r \text{ points}\}$. By \cite{Grothendieck1971}, Corollary 2.12, $$\pi_1^{p'}(\A^1_{\K} -\{r \enspace \K\text{-points}\}) \cong \pi_1^{p'} (\A^1_{\mathbb{C}} -\{r \enspace \mathbb{C}\text{-points}\}),$$ and the latter  is the maximal prime-to-$p$ quotient of the profinite completion of $\pi_1^{top}(\A^1-B, b_0)$. 
	Finally, the upshot is that since $p\not\big\vert \#G$ for all subgroups $G \subset \mathfrak{S}_{n+1}$, the mod-$p$ reduction of the topological finite covers constructed in \eqref{fundemantalHom}, gives us degree $n+1$ self-maps of $\A^1_{\K}$ with ramification indices specified in the statement of the claim.
	
	\hfill $\square$
\end{proof}
\noindent The proof of the next lemma is similar to that of Case 3 in Proposition \ref{VBness}. For notation and definitions, recall \eqref{CommPoset} and \eqref{iota_n} from Section \ref{Section2.2}.
\begin{lemma}\label{nonempty}
	Let $n$ be a  positive integer and let $\lambda\in \P_n$. Let $\sigma_n(\lambda)\in \mathfrak{P}$ be given by the following data:\begin{align*}
	& \text{positive integers } k_1,\ldots, k_r, \\ &\text{integers } e^i_j \geq 2, \text{ for each } 1\leq i\leq k_j, \text{ and } 1\leq j \leq r, 
	\end{align*} Then $S(\lambda)$ is non-empty for all $\lambda \in \P_n$ that satisfy the condition $n - \sum (e_i-1) \geq \sum (k_j - 1) $.
\end{lemma}
\begin{proof}
	Here, we continue using notation from \eqref{FR}. Our goal is to  show that $\pi(S(\lambda))$ is non-empty, i.e. there exists $\phi\in\M$ such that \begin{enumerate}[label=(\roman*)]
		\item  $\phi$ has at least $r$ branch points, say $\{b_1, \ldots, b_r \ldots\} $,
		\item $\phi^{-1} (b_j) = \{a^1_j, \ldots, a^{k_j}_j \}$, and
		\item $v_{\phi} (a^i_j) = e^i_j$ for all $1\leq i \leq k_j$ and $1\leq j\leq r$ such that $n - \sum (e^i_j-1) \geq \sum (k_j - 1)$.
	\end{enumerate}See Figure \ref{Fig3} for a schematic of the morphism $\phi$. Our proof hinges on induction on the set of branch points. The "base case" is the following: we prove the statement when $\I(\pi(S(\lambda)))$ contains polynomials such that all but one branch point have exactly one ramification point in its preimage. To this end, we show that if $\phi \in \M$ is such that \begin{gather*}
	\phi'= (x-a_1)^{e_1-1}\ldots(x-a_r)^{e_r-1}, \text{ and } \phi(a_i) \neq \phi(a_j) \text{ for } i<j,
	\end{gather*} then there exists a polynomial $\widehat{\phi}$ such that $v_{\widehat{\phi}}(a_i) = e_i$ for all $2\leq i\leq r$ and $\widehat{\phi}^{-1}(\phi(a_1)) \supset \{a^1_1,\ldots, a^l_1\}$ such that $v_{\widehat{\phi}}(a^l_1) = d_l$ and $\sum_{i=1}^{l} d_l = e_1$. In other words, to prove that the locally closed subset  $\I(\pi(S(\lambda)))$ of $\M$ is non-empty, we get hold of a generic point $\widehat{\phi} \in \I(\pi(S(\lambda)))$ by proving the existence of $\phi$ in the Zariski closure of $\I(\pi(S(\lambda)))$.  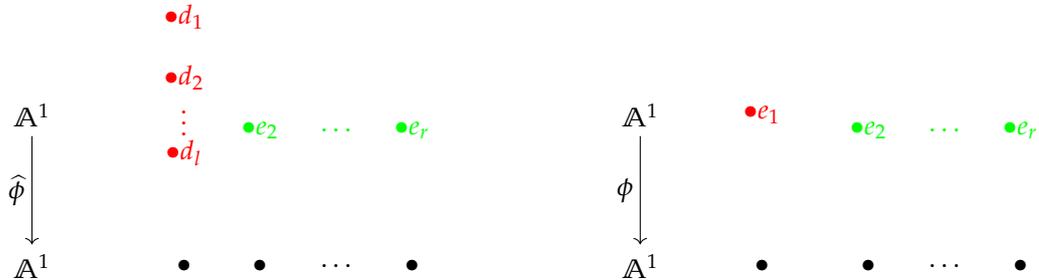
\begin{figure}[H]	\begin{tikzpicture}
	\node at (4,0.8) {\color{red}$\bullet$$d_1$}; 
	\node at (4,0) {\color{red}$\bullet$$d_2$}; 
	%\node at (1,0.2) {\color{red}$\bullet$$e^2$}; 
	\node at (4,-0.5) {\color{red}$\vdots$};    \node at (11.6,-0.5) {\color{red}$\bullet e_1$};               
	%	\node at (1,-0.3) {\color{red}$\vdots$};                
	\node (A1) at (2,-0.5) {$\A^1$};            	\node (A3) at (10,-0.5) {$\A^1$};           
	\node (A2) at (2,-2.5) {$\A^1$};    	\node (A4) at (10,-2.5) {$\A^1$};            
	\draw [->] (A1)-- (A2);                         	\draw [->] (A3)-- (A4);  \node at (9.8,-1.5) {$\phi$};    
	\node at (4,-1) {\color{red}$\bullet$$d_{l}$};                    
	\node at (1.8,-1.5) {$\widehat{\phi}$};                          \node at (5,-2.5) {$\bullet$}; \node at (6,-2.5) {$\ldots$};    \node at (7,-2.5) {$\bullet$};
	\node at (5,-0.7) {\color{green}$\bullet e_2$};  \node at (6,-0.7) {\color{green}$\ldots$}; 	\node at (13,-0.7) {\color{green}$\bullet e_2$};  \node at (14,-0.7) {\color{green}$\ldots$};  \node at (14,-2.5) {$\ldots$}; 
	\node at (7,-0.7) {\color{green}$\bullet e_r$}; 	\node at (15,-0.7) {\color{green}$\bullet e_r$};	\node at (15,-2.5) {$\bullet$};
	\node at (4,-2.5) {$\bullet$};   	\node at (13,-2.5) {$\bullet$};         \node at (11.6,-2.5) {$\bullet$};          
	\end{tikzpicture}\caption{The two schematics above represent two ramification types (for the definition, see \eqref{Mu-data}). On the left is $\widehat{\phi}$, a generic point in $\I(\pi(S(\lambda)))$, and on the right is $\phi$, a point in the closure of $\I(\pi(S(\lambda)))$.}
\end{figure} \noindent Since morphisms of $\M$ are considered up to translation, we can, without loss of generality, prove this statement on the assumption that $a_1=0$.
	
	Now, fix $r-1$ points $a_2,\ldots, a_r\in \A^1$ such that no two are equal and $a_i \neq 0$ for all $2\leq i\leq r$. 	Consider the variety  \begin{gather*}X_{e} := \big\{g\in \M: v_g(a_i) =e_i \text{ for some } a_i\in \mathbb{A}^1,  2\leq i\leq r, \text{ such that } g(a_i) \neq g(a_j) \text{ for } i<j\big\}
	\end{gather*} where $e$ is given by $e-1 := n - \sum(e_j-1)$. Then $X_{e} \cong \A^{e-1}$ by Proposition \ref{VBness}. 
	
	\noindent Let  $X'_e: = \mathcal{D}(X_e) \subset M'_n$, where $M'_n$ is as defined in \eqref{Disom}. Let  $M'_n = \spec \K[s_1, \ldots, s_n]$, where $s_1\ldots, s_n$ denote the coefficients of monic degree $n$ polynomials Then $X'_e$ is cut-out by hyperplanes given by equations $D^if(a_j) = 0$ where $1\leq j\leq  r$ and $1\leq i \leq e_j-1$. Note that $X'_e$ can be described by parametric equations in variables $t_1,\ldots, t_{e-1}$, determined by  the relation \begin{gather}\label{eq5.11}
	x^n+s_1x^{n-1}+\ldots s_{n-1}x+s_n = (x^{e-1}+t_1x^{e-1}+ \ldots+ t_{e-1})(x-a_2)^{e_2-1}\ldots (x-a_r)^{e_r-1}
	\end{gather}
	In other words, we have a linear embedding of affine spaces $$ \spec \K[t_1,\ldots, t_{e-1}]\cong \A^{e-1} \xrightarrow{T} \A^n$$ induced by equation \eqref{eq5.11} (by comparing the powers of $x$ on both sides).
	Now, we show that given $d_1,\ldots, d_l$ one can find $\alpha_2, \ldots, \alpha_l$ such that there exists $$f(x) = x^{d_1-1}(x-\alpha_2)^{d_2-1}\ldots(x-\alpha_l)^{d_l-1}(x-a_2)^{e_2-1}\ldots (x-a_r)^{e_r-1} h(x),$$ $f\in M'_n$ (in fact, $f\in X'_e$), satisfying $\I (f) (\alpha_i) = \I(f)(0)$ for all $2\leq i\leq l$, for some degree $l-1$ polynomial $h(x)$.  If we consider all possible  monic degree $l-1$ polynomials in place of $h(x)$, then $f(x)$ is still in $X_e$, except that the condition $\I (f) (\alpha_i) = \I(f)(0)$  may not be satisfied. The space of monic degree $l-1$  polynomials is $\spec \K[u_1,\ldots, u_{l-1}]$, where the coordinates are given by the coefficients.  The relations $\I (f) (\alpha_i) = \I(f)(0)$ gives $l-1$ linear conditions on $\spec \K[u_1,\ldots, u_{l-1}]$. It suffices to check that there exists $\alpha_2,\ldots, \alpha_l$ such that intersection of the $l-1$ hyperplanes whose equations are given by the linear conditions $\I (f) (\alpha_i) = \I(f)(0)$, is non-empty. Indeed, it is easy to see that for a generic choice of $\alpha_2,\ldots, \alpha_l$, the none of the equations of the hyperplanes is a scalar multiple of the other, so their intersection is forced to be non-empty, and this completes the proof of the base case. The inductive step now involves splitting the ramification point with index $e_j$, where $j\geq 2$ the same as above, and this completes the proof of the lemma.
	
	\hfill$\square$\end{proof}

\noindent 
Recall the definitions of $\sigma_n$ and $\mathfrak{P}$ from \eqref{CommPoset} and Definition \ref{Mu-data}. Then, Proposition \ref{VBness} and Lemma \ref{nonempty} imply the following.
\begin{cor}
	Let $n$ be a  positive integer and let $\lambda\in \P_n$. Let $\sigma_n(\lambda)\in \mathfrak{P}$ be given by the following data:\begin{align*}
	& \text{positive integers } k_1,\ldots, k_r, \\ &\text{integers } e^i_j \geq 2, \text{ for each } 1\leq i\leq k_j, \text{ and } 1\leq j \leq r, 
	\end{align*} Then $S(\lambda)$ is non-empty for all $\lambda \in \P_n$ that satisfy the condition $n - \sum (e_i-1) \geq \sum (k_j - 1) $.
\end{cor}
As we have learnt in this section, the ramification data associated to $\lambda \in \P^n$ solely determine whether $X(\lambda)$, modulo a subgroup of $\mathfrak{S}_n$ under its natural action, is isomorphic to an affine space. 
\begin{defn}
An elements	$\mu\in \mathfrak{P}$ is said to be \textit{\textbf{affine $n$-admissible}} if for all $\lambda \in \sigma_n^{-1}(\mu)$, one has $X(\lambda)/\mathfrak{S}_{\lvert N(\lambda) \rvert}\xrightarrow{\cong} \mathbb{A}^{d}$ for some $d>0$. 
	
	Equivalently, following \eqref{Mu-data}, $\mu\in \mathfrak{P}$ is said to be \textit{\textbf{affine $n$-admissible}} if $$n - \sum (e^i_j-1) \geq \sum (k_j - 1).$$
	
	 \hfill $\square$
\end{defn}
\begin{lemma}{\label{A-admissibility}}
	Let $m$ be a positive integer. All length $m$ ramification $\mu\in \mathfrak{P}$ are affine $n$-admissible whenever $n\geq 3m$. 
\end{lemma}
\begin{proof}
	We continue with the notation and definitions set in definition \ref{Mu-data}, \eqref{FR}. As in definition \ref{Mu-data}, \begin{gather}
	m=\length(\mu) =  \sum\limits_{1\leq j \leq r} \Big(\sum\limits_{1\leq i\leq k_j}(e^i_j-1) -1\Big). \label{eq20}
	\end{gather} Our goal is to keep $m$ fixed and find the minimum $n$ such that for all values of $r$, and $k_1,\ldots, k_r$, and $e^i_j$ satisfying \eqref{eq20}, \begin{gather}
	n - \sum (e^i_j-1) \geq \sum (k_j - 1). \label{eq21}
	\end{gather} Using \eqref{eq20} one can simplify \eqref{eq21} to $n-m\geq \sum k_j$. So, we first maximize $\sum k_j$. Maximizing the number of ramification points while keeping the length $m$ fixed, entails minimizing the ramification indices. So, $e^i_j=2$ for all $1\leq i\leq k_j$ and $1\leq j\leq r$. Therefore, \eqref{eq20} reduces to $m=\sum k_j -r$. It is easy to see that for a polynomial $\phi\in \M$ with ramification length $m$ (see \eqref{eq1.1} and Definition \ref{ramdata}), the maximum number of index $2$ ramification points a $\phi$ can have is $2m$. So, \eqref{eq21} implies $n\geq 3m$. 
	
	\hfill$\square$
\end{proof}
\begin{remark}
	Note that affine $n$-admissibility implies combinatorial $n$-admissibility. The converse is not, however, true. One can extend the proof of Lemma \ref{nonempty} to prove that elements of $\mathfrak{P}$ which are combinatorially $n$-admissible correspond to non-empty strata, and which would then lift the restriction posed by the inequality $n - \sum (e^i_j-1) \geq \sum (k_j - 1)$ in Lemma \ref{nonempty}. However, that won't be fruitful for our purpose since. In other words, if $\lambda \in \P_n$ does not satisfy \eqref{eq21}, then, even if $X(\lambda)$ is non-empty, its geometry remains unknown.
\end{remark}	\section{Spectral sequences and computation of $H^i(U)$}\label{Section5}
To work over algebraically closed fields of all characteristics at the same time, we first set up some notation and some conventions that we will use for the rest of the paper.
\begin{notation}
Let $\mathbf{Q}$ denote  $\mathbb{Q}$, the field of rational numbers, or $\mathbb{Q}_{\ell}$, the field of $\ell$-adic numbers. Throughout this section, for any $\mathbb{Z}$-scheme $V$, we continue to denote its base change to any algebraically closed field $\K$ by $V$. In turn, we mean $H^*(V;\mathbf{Q})$  to stand for both $H^*(V(\mathbb{C});\mathbb{Q})$ as well as $H_{\et}^*(V_{/\K}; \mathbb{Q}_{\ell})$. 

Furthermore, when $V$ is $\U^m_n$, its $\mathfrak{S}_n$-quotient $\S^m_n$ or $\Simp^m_n$ for some positive integer $m$, we will always assume that $n+1 <char \K$ whenever $char \K>0$. 
We fix a positive integer $m$ for the rest of the paper, and a positive integer $n$ that satisfy $n\geq 3m$.

\hfill$\square$\end{notation}
\begin{remark}\label{remark6.2}
	The assumption $n\geq 3m$ is required because it is a sufficient condition for Lemma \ref{A-admissibility}, which in turn is an important ingredient in the proof of Theorem \ref{thmB}. However, in the case when $m=1$, we have well-known answers for $n<3$ (compare with Remark \ref{rem1.3}). When $n=2$, we have $$D: \mathit{Simp}_2(\mathbb{C}) \to Conf_2(\mathbb{C})$$ is an isomorphism (where $Conf_2(\mathbb{C})$ denotes the unordered configuration space of two points in $\mathbb{C}$). Arnol'd's work (see e.g. \cite{Arnol'd1969}) answers completely the cohomology of $Conf_2(\mathbb{C})$. When $n=1$, the result is trivial because all morphisms are simply-branched.
\end{remark}
In this section, we construct a cohomology spectral sequence $E_{\bullet}^{\bullet,\bullet}$ that converges to $H^*(\U^m_n;\Q)$. To obtain $H^*(\S^m_n;\Q) \cong \big(H^*(\U^m_n;\Q)\big)^{\mathfrak{S}_n}$, we take the $\mathfrak{S}_n$ invariants of $E_{\bullet}^{\bullet,\bullet}$ and show that the resulting spectral sequence, which converges to $H^*(\Simp^m_n;\Q)$, degenerates on the $E_1$ page. First, we start with the following lemma.
\begin{lemma}\label{lem5.1}
Let $\K$ be an algebraically closed field. Let $m\geq 1$, and $n\geq 3$, satisfy $n+1<char \K$ whenever $char \K>0$. The complex $\mathit{A}^{\bullet}$ given by
\begin{equation}
\begin{gathered}\label{resol}\begin{split}
\mathbf{Q}_{\X_n} \to \bigoplus_{\substack{l^{m}(\lambda) = 1,\\ \lambda \in \P^m_n}} (i_{\lambda})^{*}\mathbf{Q}_{X(\lambda)} \to \bigoplus_{\substack{l^{m}(\lambda) = p,\\ \lambda \in \P^m_n}}\widetilde{H}^0 (\widehat{0},\lambda) \otimes(i_{\lambda})^{*} \mathbf{Q}_{X(\lambda)}    \to \ldots \\ \to \bigoplus_{\substack{l^{m}(\lambda) = p,\\ \lambda \in \P^m_n}} {\widetilde{H}}^{p-2} (\widehat{0},\lambda) \otimes (i_{\lambda})^{*}\mathbf{Q}_{X(\lambda)} \to \ldots
\end{split}\end{gathered}
\end{equation}
is quasi-isomorphic to $j_{!}\mathbf{Q}_{\U^m_n} $, where $j$ denotes the inclusion of the open stratum $\mathcal{U}^m_n \hookrightarrow \X_n$ and for each $\lambda\in \P^m_n$, the map $i_{\lambda}: X(\lambda) \hookrightarrow \X_n$ is an inclusion of closed strata.
\end{lemma}	

\begin{proof}
Following ~\cite[Section 3]{Petersen2017}, let $\mathcal{F}^{\bullet}$ be the complex of sheaves on $\X_n$ defined by $$\mathcal{F}^{p} =  \bigoplus_{\substack{l^m(\lambda) \geq p,\\ \lambda \in \P^m_n}} {\widetilde{C}}^{p-2} (\widehat{0},\lambda) \otimes (i_{\lambda})^{*} \mathbf{Q}_{X(\lambda)} $$
where ${\widetilde{C}}^{p-2} (\widehat{0},\lambda)$ is as defined in \Cref{Section3.1}.
That $\mathcal{F}^{\bullet}$ gives a resolution of $j_{!}\mathbf{Q}_{\U^m_n} $ follows from ~\cite{Petersen2017}, or more simply, just by using the inclusion-exclusion principle. Finally, note that $\mathcal{F}^{\bullet}$ carries a filtration by the length of elements in $\P^m_n$, which in turn gives a quasi-isomorphism to \eqref{resol} once we incorporate Proposition \ref{semimodular}. 

\hfill $\square$
\end{proof}

 \noindent We now prove Theorem \ref{thmB} and \ref{corB}.
 
\begin{proof}  We fix a positive integer $n\geq 3$. Let $m$ be a positive integer that satisfy $n\geq 3m$. % Furthermore, we assume $n\geq 3m$, so that all length $m$ ramification (see \eqref{Mu-data} and \eqref{lengthp}) are affine $n$-admissible (see Lemma \ref{A-admissibility}).
  	The variety $\Simp^m_n$ is a Zariski dense open subset of $\M \cong \A^n$, and hence connected. So, $H^0(\Simp^m_n;\mathbf{Q}) \cong \mathbf{Q}$.
Now, continuing with the resolution in  \eqref{resol}, we construct a second quadrant double complex $K^{\bullet, \bullet}$ by taking the global Verdier dual of the complex in \eqref{resol}. If $\mathbf{Q}_{\X_n} \hookrightarrow I^{\bullet}$ is an injective resolution of $\mathbf{Q}_{\X_n}$-modules, then  $$K^{\bullet,\bullet} = RHom(\mathit{A}^{\bullet}, \mathbf{Q}_{\X_n})$$ where $$K^{-p,q} = Hom(\mathit{A}^{p}, I^q).$$  For each $p$, take the naive filtration $\tau_{\geq q}$ on $K^{-p,\bullet}$ via $$\Big(\tau_{\geq q}\big( K^{-p,\bullet} \big)\Big)^i= \Bigg\{ \begin{array}{lr}
0 & \text{for } i<q, \\
K^{-p,q}& \text{for } i\geq q.
\end{array}	$$ Thus, we obtain a spectral sequence which reads as $$E_1^{-p,q} = Ext^q(\mathit{A}^p, \mathbf{Q}_{\X_n}) \implies Ext^q(j_{!}\mathbf{Q}_{\U^m_n}, \mathbf{Q}_{\X_n})\cong H^q(\U^m_n; \mathbf{Q})$$
The last isomorphism above is implied by the fact that $(j_!, j^*)$ is an adjoint pair. Moreover, all morphisms considered in this paper are algebraic, so this is a spectral sequence of mixed Hodge structures. Now we take the $\mathfrak{S}_n$ invariants of each term on the $E_1$ page (again, the transfer map being algerbaic respects the mixed Hodge structures.) \begin{gather}
\Big( Ext^q(\mathit{A}^p, \mathbf{Q}_{\X_n})\Big)^{\mathfrak{S}_n} \nonumber\\= \bigg(\bigoplus\limits_{l^m(\lambda) = p} {\widetilde{H}}_{p-2} (\widehat{0},\lambda) \otimes Ext\Big((i_{\lambda})^{*}\mathbf{Q}_{X(\lambda)}, \mathbf{Q}_{\X_n}\Big) \bigg)^{\mathfrak{S}_n} \label{isom20}\\ \cong \bigg(\bigoplus\limits_{l^m(\lambda) = p} {\widetilde{H}}_{p-2} (\widehat{0},\lambda) \otimes H^q(\X_n, \X_n - X(\lambda))\bigg)^{\mathfrak{S}_n}\label{isom21}
\end{gather} 
The isomorphism between \eqref{isom20} and \eqref{isom21} follows from the fact $$Ext^q\Big((i_{\lambda})^{*}\mathbf{Q}_{X(\lambda)}, \mathbf{Q}_{\X_n}\Big) \cong  H^q(\X_n, \X_n - X(\lambda);\mathbf{Q})$$ because of the distinguished triangle: \[
\begin{tikzcd}
& Rj_!j^*  \arrow{dr} \\
 Ri_{\lambda}^*i_{\lambda}^* \arrow{ur}{[1]}  \arrow{rr}&& id_{\X_n} 
\end{tikzcd}
\]
\noindent To study each term of the spectral sequence, we need to compute \begin{enumerate}[label=(\roman*)]
	\item   $H^q(\X_n, \X_n - X(\lambda);\mathbf{Q})$, and
	\item $ \Big(\widetilde{H}_{p-2} (\widehat{0},\lambda)\Big)^{\mathfrak{S}_n}$.
\end{enumerate}  
\noindent For (ii),	we first, we consider the case when $m=1$, and study the action of $\mathfrak{S}_{n}$ on $\widetilde{H}_{p-2} (\widehat{0},\lambda)$. This, in turn, is based on the action of $\mathfrak{S}_n$ on $\Pi_n$, the partition lattice on $\{1,2\ldots, n\}$ which is completely known and well-documented in \cite{Wachs2006}. We  show that  \begin{gather}
\Big(\widetilde{H}_{p-2} (\widehat{0},\lambda)\Big)^{\mathfrak{S}_n} =0    \text{ for all } \lambda \in \P_n, l(\lambda) \geq 2. \label{length2vanish}
\end{gather} Suppose there exists $0\neq \omega \in \Big(\widetilde{H}_{p-2} (\widehat{0},\lambda)\Big)^{\mathfrak{S}_n}$, i.e. $\omega$ is a $\mathbf{Q}$-linear combination of $(p+1)$-chains starting at $\widehat{0}$ and ending at $\lambda$, that is invariant under the action of $\mathfrak{S}_{n}$. Recall that $\P_n\subset \Pi_n\times \Pi_n$. Let  $proj_i$ denote the projection of $\P_n$ to the $i^{th}$ copy of $\Pi_n$, for $i=1,2$. 
A simple, but crucial observation is that, if $\length_{\Pi_n}$ denotes the length function on $\Pi_n$, then $\length_{\Pi_n}(proj_2(\lambda)) = l(\lambda)$. In fact, using the definitions and notation set up in \eqref{rho} and \eqref{FR}, one has $$\length_{\Pi_n}(proj_1(\lambda)) = \sum_{i,j}(e^i_j-2)$$ and $$\length_{\Pi_n}(proj_2(\lambda)) = \sum\limits_{1\leq j \leq r} \Big(\sum\limits_{1\leq i\leq k_j}(e^i_j-1) -1\Big)  =l(\lambda) = p.$$
Therefore, $proj_2(\omega)$ is a nonzero $\mathfrak{S}_n$-invariant element in $\widetilde{H}_{p-2} (proj_2(\widehat{0}),proj_2(\lambda))$, where $proj_2(\widehat{0})$ is the $\widehat{0}_{\Pi_n}$ of $\Pi_n$,  i.e. the minimal element of the geometric lattice $\Pi_n$. But this contradicts the well-known fact that $\Big(\widetilde{H}_{p-2} (\widehat{0}_{\Pi_n},\lambda')\Big)^{\mathfrak{S}_n} =0$ for all $\lambda' \in \Pi_n$ of length $p$, and in particular, for $\lambda' = proj_2(\lambda)$. Following the proof of \eqref{length2vanish}, one has, for  $m\geq 2$, \begin{gather}
\Big(\widetilde{H}_{p-2} (\widehat{0},\lambda)\Big)^{\mathfrak{S}_n} =0 \text{ for all } \lambda \in \P^m_n, l^{m}(\lambda) \geq 2. \label{E1q}
\end{gather} 
The discussion on (ii) above implies that $E_1^{-p,q} =0$ for all $q$ whenever $p\geq 2$. When $p=1$, for each $\lambda \in \P_n$ of length $m+1$, which is equivalent to saying $\lambda \in \P^m_n$ of length $1$, we have \begin{gather}
E_1^{-1,q} = \bigoplus_{\substack{l^{m}(\lambda) = 1,\\ \lambda \in \P^m_n}}  H^q(\X_n, \X_n - X(\lambda);\mathbf{Q}).\end{gather}

\noindent For (i), note that for an arbitrary $\lambda \in \P_n$ we have already seen that $X(\lambda)$, in general, has singularities. But when $m$ is a positive integer and $n\geq 3m$, for all $\lambda \in \P_n$ satisfying $l(\lambda) =m$, Proposition \ref{VBness}, Lemma \ref{nonempty} and Lemma \ref{A-admissibility} imply that $X(\lambda)$ is non-empty and $$X(\lambda)/\mathfrak{S}_{N(\lambda)} \cong \A^{n-m}.$$ For the rest of the proof, we fix an integer $n$ that satisfies $n\geq 3m$ and $n< char \K -1$ whenever $char \K>0$. So now, taking $\mathfrak{S}_n$ invariants of \eqref{E1q}, one obtains: 

\begin{flalign} &\Big(E^{-1,q}\Big)^{\mathfrak{S}_n}  \nonumber \\ \nonumber \\ &=
\bigoplus_{\substack{\{\mu \in \mathfrak{P}: \\\length(\mu) = m,\\   \widetilde{\mu} \text{ a choice of a lift of } \mu\}}} H^q\Big(\X_n/{\mathfrak{S}_{\lvert N(\widetilde{\mu}) \rvert}} , \X_n/{\mathfrak{S}_{\lvert N(\widetilde{\mu}) \rvert}} - X(\widetilde{\mu})/{\mathfrak{S}_{\lvert N(\widetilde{\mu}) \rvert}};\mathbf{Q}\Big) \label{isom5.7}\\ \nonumber \\ & \cong \bigoplus_{\substack{\{\mu \in \mathfrak{P}: \\\length(\mu) = m,\\   \widetilde{\mu} \text{ a choice of a lift of } \mu\}}}  H^{q-2m}(X(\widetilde{\mu})/{\mathfrak{S}_{\lvert N(\widetilde{\mu}) \rvert}};\mathbf{Q}) \label{isom5.8}\\ \nonumber\\& \cong
\begin{cases}
0        & \text{if } q \neq 2m  \\
\mathbf{Q}(-m)^{\oplus\mathbf{c}(m)}        & \text{if } q=2m, 
\end{cases} \label{isom5.9}
\end{flalign}
where $\mathbf{c}(m) = \#\{\mu\in \mathfrak{P}: \length(\mu)  = m\}$, a positive integer defined in \eqref{c(m)}. For the last three steps above, note the following:  \begin{enumerate}
	\item  $\X_n/{\mathfrak{S}_{\lvert N(\widetilde{\mu}) \rvert}} \cong \A^n$; this is because $\X_n \cong \A^n$ and $\mathfrak{S}_{\lvert N(\widetilde{\mu}) \rvert}$ is a subgroup of $\mathfrak{S}_n$ that acts by permuting the coordinates.
	\item We know from Proposition \ref{VBness} that $X(\widetilde{\mu})/{\mathfrak{S}_{\lvert N(\widetilde{\mu}) \rvert}}$ is a smooth codimension $m$ closed subvariety in  $\X_n/{\mathfrak{S}_{\lvert N(\widetilde{\mu}) \rvert}}$ ,  so by the Gysin homomorphism (see e.g. \cite[Theorem 16.1]{MilneVersion2.21March222013}) we obtain \eqref{isom5.8} from \eqref{isom5.7}.
	\item By Proposition \ref{VBness}, $X(\widetilde{\mu})/{\mathfrak{S}_{\lvert N(\widetilde{\mu}) \rvert}}\cong \A^{n-m}$, which gives us \eqref{isom5.9} from \eqref{isom5.8}.
\end{enumerate}   This completes of the proof of Theorem \ref{thmB}. 
Finally, let $char\K=p$ and let $q=p^d$ for some positive integer $d$. The Grothendieck-Lefschetz trace formula now reads as  \begin{gather}
\#	\Simp^m_n(\F_q) = q^n \sum_{i} (-1)^i \Trace (\Frob_q:H^i(\Simp^m_n;\Q_{\ell})).\label{6.10}
\end{gather} By equation \eqref{isom5.8}, the right-hand-side of \eqref{6.10} equals $q^n-\mathbf{c}(m)q^{n-m}$, thus proving
Corollary \ref{corB}.

\hfill$\square$
\end{proof}

	\bibliographystyle{alpha}
\bibliography{SIMP}

\begin{thebibliography}{HOPS17}

\bibitem[Arn69]{Arnol'd1969}
V.I. Arnol'd.
\newblock The cohomology of the colored braid group.
\newblock {\em Mat. Zametki, 5:227–231}, 1969.

\bibitem[Arn70]{V.I.Arnold1970}
V.I. Arnol'd.
\newblock On some topological invariants of algebraic functions.
\newblock {\em Tr. Mosc. Mat. Obsc., pages 27–46}, 1970.

\bibitem[Bj{\"o}80]{BjoernerJuly1980}
Anders Bj{\"o}rner.
\newblock Shellable and cohen-macaulay partially ordered sets.
\newblock {\em Transactions of the American Mathematical Society Vol. 260, No.
  1, pp. 159-183}, July, 1980.

\bibitem[Cle72]{Clebsch1872}
A.~Clebsch.
\newblock Zur theorie der riemann’schen flachen.
\newblock {\em Mathematische Annalen, 6:216–230}, 1872.

\bibitem[EEHS91]{Eisenbud1991}
David Eisenbud, Noam Elkies, Joe Harris, and Robert Speiser.
\newblock On the hurwitz scheme and its monodromy.
\newblock {\em Compositio Mathematica, tome 77, no 1, p. 95-117}, 1991.

\bibitem[EVW15]{Ellenberg2015}
Jordan~S. Ellenberg, Akshay Venkatesh, and Craig Westerland.
\newblock Homological stability for hurwitz spaces and the cohen-lenstra
  conjecture over function fields.
\newblock {\em https://arxiv.org/abs/0912.0325, to appear in Annals of
  Mathematics}, 2015.

\bibitem[Ful69]{Fulton1969}
W.~Fulton.
\newblock Hurwitz schemes and moduli of curves.
\newblock {\em Annals of Mathematics, 90:542–575}, 1969.

\bibitem[FW15]{Farb2015}
Benson Farb and Jesse Wolfson.
\newblock Étale homological stability and arithmetic statistics.
\newblock {\em https://arxiv.org/abs/1512.00415}, 2015.

\bibitem[GR71]{Grothendieck1971}
Alexander Grothendieck and Michele Raynaud.
\newblock Revêtements étales et groupe fondamental.
\newblock {\em Seminaire de Geometrie Algebrique du Bois Marie 1960/61 (SGA
  1)}, 1971.

\bibitem[Gro66]{Grothendieck1966}
Alexander Grothendieck.
\newblock Éléments de géométrie algébrique : Iv. Étude locale des
  schémas et des morphismes de schémas, troisième partie.
\newblock {\em Publications Mathématiques de l'IHÉS, Volume 28}, 1966.

\bibitem[Har77]{Hartshorne1977}
Robin Hartshorne.
\newblock {\em Algebraic Geometry}.
\newblock Springer, 1977.

\bibitem[HOPS17]{Stevenson2017}
David Harbater, Andrew Obus, Rachel Pries, and Katherine Stevenson.
\newblock Abhyankar's conjectures in galois theory: Current status and future
  directions.
\newblock {\em https://arxiv.org/abs/1408.0859, To appear in Bull. Amer. Math.
  Soc}, 2017.

\bibitem[Lan02]{Lang2002}
Serge Lang.
\newblock {\em Algebra, Third Edition}.
\newblock Springer-Verlag, 2002.

\bibitem[Mil13]{MilneVersion2.21March222013}
James Milne.
\newblock Lectures on Étale cohomology.
\newblock Version 2.21 March 22, 2013.

\bibitem[Nap98]{Napolitano1998}
F.~Napolitano.
\newblock Topology of complements of strata of the discriminant of polynomials.
\newblock {\em Comptes Rendus de l'Académie des Sciences - Series I -
  Mathematics Volume 327, Issue 7, Pages 665-670}, 1998.

\bibitem[Pet17]{Petersen2017}
Dan Petersen.
\newblock A spectral sequence for stratified spaces and configuration spaces of
  points.
\newblock {\em Geometry and Topology, Volume 21, Issue 4}, 2017.

\bibitem[RW06]{Romagny2006}
Matthieu Romagny and Stefan Wewers.
\newblock Hurwitz spaces.
\newblock {\em Śeminaires \& Congrés 13, p. 313–341}, 2006.

\bibitem[Wac06]{Wachs2006}
Michelle~L. Wachs.
\newblock Poset topology: Tools and applications.
\newblock {\em Geometric Combinatorics}, 2006.

\end{thebibliography}
		\end{document}